\documentclass[11pt]{article}

\usepackage{epsfig,epsf,fancybox}
\usepackage{mathrsfs}
\usepackage{graphicx,graphics}
\usepackage{url}
\usepackage{epsf,epstopdf}
\usepackage{color}
\usepackage{amsmath,amsxtra,amsfonts,amscd,amssymb,bm}
\usepackage{multirow,cite}
\usepackage[algo2e,linesnumbered,vlined,ruled]{algorithm2e}

\numberwithin{table}{section}
\newtheorem{theorem}{Theorem}[section]
\newtheorem{corollary}[theorem]{Corollary}
\newtheorem{lemma}[theorem]{Lemma}
\newtheorem{proposition}[theorem]{Proposition}
\newtheorem{definition}[theorem]{Definition}

\newtheorem{assumption}[theorem]{Assumption}

\newcommand{\be}{\begin{equation}}
\newcommand{\ee}{\end{equation}}
\newcommand{\bee}{\begin{equation*}}
\newcommand{\eee}{\end{equation*}}
\newcommand{\bea}{\begin{eqnarray}}
\newcommand{\eea}{\end{eqnarray}}
\newcommand{\beaa}{\begin{eqnarray*}}
\newcommand{\eeaa}{\end{eqnarray*}}

\newcommand{\st}{\textrm{s.t. }}

\newcommand{\R}{\mathbb{R}}
\newcommand{\cM}{\mathcal{M}}

\newcommand{\cT}{\mathcal{T}}

\newcommand{\cE}{\mathcal{E}}

\newcommand{\grad}{\mathrm{grad}}
\newcommand{\Proj}{\mathrm{Proj}}
\newcommand{\Hess}{\mathrm{Hess}}
\newcommand{\J}{\mathrm{J}}
\newcommand{\Retr}{\mathrm{Retr}}
\newcommand{\St}{\mathrm{St}}

\newcommand{\vvec}{\mathrm{vec}}

\newcommand{\half}{\frac{1}{2}}

\newcommand{\tanst}{\cT_X\St_{n,r}}

\begin{document}

%\title{Iteration Complexity of Some First-Order Algorithms for Computing an $\epsilon$-Stationary Point of Structured Nonconvex Optimization Problems}

%\title{Multi-Block Optimization over Riemannian Manifolds: \\ an Iteration Complexity Analysis}

\title{A Cubic Regularized Newton's Method over Riemannian Manifolds}

\author{
	Junyu Zhang\thanks{Department of Industrial \& System Engineering, University of Minnesota (zhan4393@umn.edu).}
	\and
	Shuzhong Zhang\thanks{Department of Industrial \& System Engineering, University of Minnesota (zhangs@umn.edu).}
}

\date{\today}

\maketitle

\begin{abstract}
	In this paper we present a cubic regularized Newton's method to minimize a smooth function over a Riemannian manifold. The proposed algorithm is shown to reach a second-order $\epsilon$-stationary point within $\mathcal{O}(1/\epsilon^{\frac{3}{2}})$ iterations, under the condition that the pullbacks are locally Lipschitz continuous, a condition that is shown to be satisfied if the manifold is compact.
	Furthermore, we present %if the objective function satisfies some additional conditions, %some more specific conditions,
	a local superlinear convergence result under some additional conditions.
\end{abstract}

\vspace{0.25cm}

\noindent {\bf Keywords:}
Riemannian optimization, Stiefel Manifold, second-order $\epsilon$-stationary solution, cubic regularization, iteration complexity.

\vspace{0.5cm}

\section{Introduction} \label{sec:intro}
Optimization over a Riemannian manifold (e.g.\ Stiefel manifolds) % or products of Stiefel manifolds,
is an important model with numerous applications, including: the PCA, the sparse PCA, eigenvalue and combinatorial optimization; see \cite{GeoStif,APP:RM_sto4,APP:SDP_Gro_IEQ,APP:nemirv,APP:Stie-products}. %In general, these  problems can be written as
Specifically, this paper considers the following model
\be
\label{prob:Riemannian_Opt}
\min \,\,\, f(x),\,\, \mbox{ subject to } x\in\cM,
\ee
where $\cM$ is a Riemannian manifold. The most commonly encountered Riemannian manifolds include: the Stiefel manifold $\mathrm{St}_{n,k}:= \{X\in\R^{n\times k}: X^\top X = I_{k\times k}\}$, the Grassmann manifold $\mathrm{Gr}_{n,k}:= \mathrm{St}_{n,k}/\mathrm{St}_{k,k} $, the generalized orthogonal matrices $\{X\in\R^{n\times k}: X^\top M X = I_{k\times k}\}$ with $M\succ 0$, the sphere $\mathrm{S}^{n-1}:=\{x\in\R^n:\|x\|_2=1\}$, the low-rank elliptope $\{X\in\R^{n\times n}: \mathrm{diag}(X) = \mathbf{1}_n, X\succeq 0, \mathrm{rank}(X)\leq r\leq n\}$, the low-rank spectrahedron $\{X\in\R^{n\times n}: \mathrm{tr}(X) = 1, X\succeq 0, \mathrm{rank}(X)\leq r\leq n\}$, and a Cartesian product
of the above. %of these manifolds.

%These problems are
Model~\eqref{prob:Riemannian_Opt} is
not only non-convex from a Euclidean standpoint, but also difficult to preserve feasibility once an iterate steps out of $\cM$.
An effective way to resolve such difficulty under the framework of Riemannian optimization is to incorporate a so-called retraction operation, which gets the iterate back to the manifold in close proximity.
%One goal of Riemannian optimization is to study the properties of {\it retraction}, which is an operation to get the iterate back onto the manifold in close proximity.
%One way to deal with such difficulty is through the Riemannian optimization techniques. E.g., the authors of
Hence, gradient-type approach with retraction becomes a natural choice as solution method for Riemannian optimization. Numerous studies have been conducted along this line;
%As an example, \cite{wenDouble} considered optimization with orthogonality constraints by incorporating a curvilinear line-search in a Cayley transform-based first-order method. Similar approaches include more general first-order methods based on retractions;
cf.\ \cite{Stief:SVRG-Jiang,Stief:SVRG-Kasai,wenDouble}. Note that  such first-order methods typically assure an iteration complexity of $\mathcal{O}(1/\epsilon^2)$ to reach a {\it first-order}\/ $\epsilon$-stationary point, without guaranteeing any second-order optimality condition.
%Though the very
Speaking of which, on the positive side, recent results (such as the ones in \cite{Dynamical-system-1,Dynamical-system-2,Dynamical-system-3,Dynamical-system-4}) show that many first-order methods, including the Riemannian gradient descent method, actually converge to a strict saddle point with probability 0 if initialized at random.
Yet, there is no guarantee of iteration complexity of convergence to a second-order stationary point. Counter examples do exist, showing that the gradient descent algorithm may take exponential number of steps to converge to a second-order stationary point \cite{GD-Exp-Cvg}. In other words, these methods do not guarantee to reach a point satisfying the second-order optimality condition in a reasonable amount of time in the worst case. On the other hand, the second-order optimality condition actually turns out to be a key quality to possess in many applications. For example, in \cite{LRSDP-SamBurer} the authors proposed a non-convex low-rank approximation
\bea
\label{prob:LRSDP}
&\min& \langle C, VV^\top\rangle \mbox{    } ~~~~ \st V\in \R^{n\times k},  \mathcal{A}(VV^\top) = b
\eea
for the standard SDP
\bea
\label{prob:SDP}
&\min& \langle C, X\rangle \mbox{    } ~~~~ \st X\in S^{n\times n}, X \succeq 0, \mathcal{A}(X) = b.
\eea
It was shown in \cite{Gro:2} that if $V^*$ is a second-order stationary point of \eqref{prob:LRSDP} and is rank-deficient then $V^*(V^*)^\top$ is optimal to \eqref{prob:SDP}. Under a similar setting, \cite{Gro:1} presented a Grothendieck-type inequality %developed in  shows that
\be
\label{grothendieck-1}
\langle C,X^*\rangle \leq \langle C,V^*(V^*)^\top\rangle \leq \langle C,X^*\rangle + \frac{8n}{\sqrt{k}} \|C\|_2,
\ee
for the SDP relaxation of the max-cut problem, where $X^*$ is optimal to \eqref{prob:SDP} and $V^*$ is an arbitrary second-order stationary point of \eqref{prob:LRSDP}. This bound was later improved in \cite{APP:SDP_Gro_IEQ} for the low-rank max-cut-SDP (MC-SDP) problem and the max-orthogonal-cut-SDP (MOC-SDP) problem, whose feasible regions are a product of spheres and a product of Stiefel manifolds, respectively. The authors also extend the result to the second-order $\epsilon$-stationary points.
%rather than an \revise{exact one}.
%These lines of research
Such results reveal the importance of the second-order $\epsilon$-stationary solutions, thus promoting the use of the Hessian information.
In the literature, globally convergent algorithms guaranteeing second-order optimality conditions for Riemannian optimization are based on the trust-region method~\cite{RM_tru,RM_tru1,RM_glo}.
In general, the Riemannian trust-region (RTR) algorithms return a first-order $\epsilon$-stationary point within $\mathcal{O}(1/\epsilon^{2})$ iterations, and returns a second-order $\epsilon$-stationary point (to be defined in later in the paper) in $\mathcal{O}(1/\epsilon^{2.5})$ iterations.

In this paper, we propose a cubic regularized Riemannian Newton's (CRRN) method to solve the smooth optimization problems on Riemannian manifolds. This method follows the line of research originated from Nesterov's pioneering work \cite{Cubic:Nestrov}, which has drawn significant research attention in the classical Euclidean optimization context; see e.g. \cite{Cubic:Cartis-1,Cubic:Cartis-2,Cubic:Cartis-3,Cubic:Agarwal,Cubic:Carmon}. Such methods typically find a second-order $\epsilon$-stationary point within $\mathcal{O}(1/\epsilon^{\frac{3}{2}})$ iterations in the Euclidean case. In this paper, we prove that this iteration complexity result carries over to general Riemannian optimization. By assuming a certain local Lojasiewicz inequality property or the non-degeneracy of the Riemannian Hessian, local superlinear convergence can be further guaranteed.

\textbf{Organization.} In Section \ref{sec:basics}, we introduce some basic properties of Riemannian manifolds, as well as the notions that are essential for Riemannian optimization.
%Several significant and handy properties of these concepts are also derived in this section.
In Section \ref{sec:algo}, we present our algorithms and their iteration complexity bounds. In Section \ref{sec:app}, we discuss the application of
%give an example of implementing
our algorithm specialized to the Stiefel manifold, and report results of numerical experiments. % results to show the effectiveness of our method.
All the relevant constants required by the algorithm are %also figured out
explicitly computed in this special case. In Section~\ref{sec:conclusion}, we conclude the paper.

\textbf{Notations.} The Jacobian of a vector function $\mathbf{g}(x): \R^n\rightarrow\R^m$ is denoted as $\mathrm{J}\mathbf{g}(x)$. When dealing with a composite function $\mathbf{g}(\mathbf{f}(\xi))$ with $\mathbf{f}(\xi):\R^k\rightarrow\R^n$, we use $\mathrm{J}\mathbf{g}(\mathbf{f}(\xi))$ to denote the Jacobian of $\mathbf{g}$ at point $\mathbf{f}(\xi)$ and use $\mathrm{J}_{\xi}\mathbf{g}(\mathbf{f}(\xi))$ to denote the Jacobian of the function $\mathbf{g}\circ\mathbf{f}$ at point $\xi$. We also differentiate $\nabla$ and $\nabla^2$ with $\nabla_{\xi}$ and $\nabla^2_{\xi}$ if $\mathbf{g}$ is a scaler function. For a Hessian $\nabla^2f(x)$ operating along a direction $z$, we shall write $\nabla^2f(x)[z]$ instead of $\nabla^2f(x)z$, as the former is less confusing especially when $z$ itself is a matrix.

\section{Riemannian Optimization}\label{sec:basics}

This section provides some preliminaries regarding Riemannian optimization, which aims to minimize a smooth function over a Riemannian submanifold  $\cM$  of a Euclidean space $\cE$.
%, for the convenience of later discussions we shall first introduce some basic notions and properties related to a Riemannian manifold denoted as $\cM$, which is a subset of a Euclidean space $\cE$.
For an in-depth discussion of Riemannian manifolds, we refer the interested readers to \cite{Opt_Manif:Absil-etal-2009,Smooth_Manif:Lee-John-2008}. Our brief introduction is to be followed by a discussion about the optimality conditions under the Riemannian optimization setting. We then introduce the notion of the extended retraction. Finally, we introduce the pullback operation and its properties, for the benefit of analyzing the performance of our algorithms in later sections.

\subsection{Riemannian Manifolds}
Consider a differentiable submanifold in a Eucliedean space (we follow the notations of \cite{Retraction}).
%for the ease of later derivations.
\begin{definition}[Differentiable Submanifold \cite{Retraction}]
	\label{defn:dif-submanifold}
	We call $\cM$ to be a $d$-dimensional $C^k$ differentiable submanifold of $\mathbb{R}^n,$ $k\geq 2,$ if for any $x\in\cM$ there exists a neighbourhood $B_x$ of $x$ in $\mathbb{R}^n$ and a $C^k$ diffeomorphism $\psi$ on $B_x$ into $\mathbb{R}^n$ such that $\forall y\in B_x$, $y\in \cM$ if and only if
	$$\psi_{d+1}(y) = \cdots = \psi_n(y) = 0.$$
\end{definition}
A useful insight of this definition is to recognize that a submanifold can locally  be induced by a set of equations. By applying the implicit function theorem,
%it is not hard to prove the validity of the following result.
the following is readily seen:
\begin{corollary}
	Let $\cM := \{x\in\mathbb{R}^n:\:\phi(x) = 0\}$ where $\phi:\mathbb{R}^n\rightarrow\mathbb{R}^m$ is a $C^k$ smooth mapping. Then $\cM$ is an $n-m$ dimensional $C^k$ submanifold of $\mathbb{R}^n$ if for any $x\in\cM$, the Jacobian matrix $\mathrm{J}\phi(x)$ has full rank.
\end{corollary}
For example, the sphere  $\mathrm{S}^{n-1} := \{x\in\mathbb{R}^n:\: \|x\|^2 = 1\}$ is an $n-1$ dimensional $C^\infty$ submanifold and the Stiefel manifold
$\mathrm{St}_{n,r} = \{X\in\mathbb{R}^{n\times r}:\: X^\top X = I_{r\times r}\}$ is an $nr - \frac{r(r+1)}{2}$ dimensional $C^\infty$  submanifold. For a submanifold $\cM$ embedded in a Euclidean space $\cE$, the tangent space of $\cM$ at point $x\in\cM$ is denoted by $\cT_x\cM$, which can be characterized by the following subspace of $\cE$:
\begin{definition}[Tangent Space]
	Suppose $\cM$ is a submanifold of $\cE$. The tangent space of $\cM$ at $x$ is defined as
	%$$\cT_x\cM = \bigg\{\frac{d\gamma(t)}{dt}\bigg|_{t=0}: \gamma: [-\delta,\delta]\rightarrow\cE, \delta>0, \mbox{ is a smooth curve with } \gamma(0) = x, \gamma([-\delta,\delta])\subset \cM\bigg\}.$$
	$$\cT_x\cM = \bigg\{\gamma'(0): \gamma \mbox{ is a smooth curve with } \gamma(0) = x, \gamma([-\delta,\delta])\subset \cM, \delta>0\bigg\}.$$
	Then the tangent bundle is defined as
	$\cT\cM = \{(x,\xi):\: x\in\cM, \xi\in\cT_x\cM\}.$
\end{definition}
For a submanifold induced by $\cM = \{x\in\R^n:\:\phi(x) =0 \}$,  an effective way to characterize the tangent space is
\begin{equation}
\label{defn:tan-space}
\cT_x\cM = \mathrm{Ran}(\mathrm{J}\phi(x)^\top)^{\perp},
\end{equation}
namely the orthogonal complement of the range space of $\mathrm{J}\phi(x)^\top$. If the tangent spaces is equipped with an inner product (hence inducing a metric),
% and the corresponding metric,
then this submanifold is known as Riemannian.
\begin{definition}[Riemannian Submanifold]
	Suppose $\cM$ is a differentiable submanifold of $\cE$. We call $\cM$ to be a Riemannian submanifold of $\cE$, if for any $x\in\cM$ the tangent space $\cT_x\cM$ is endowed with the Euclidean inner product;
	that is, for any $\eta,\xi\in\cT_x\cM$, if we let $\cT_x\cM$ be embedded in $\cE$ as a subspace, then the inner product on $\cT_x\cM$ is defined as
	$\langle \xi,\eta\rangle_x := \langle \xi,\eta\rangle,$
	where the latter is the standard Euclidean inner product. Hence the norm $\|\cdot\|_x$ induced by $\langle\cdot,\cdot\rangle_x$ is also the same as the standard $L_2$-norm (or the Frobenius norm in the matrix case).
\end{definition}

Roughly speaking, a Riemannian manifold is a differentiable manifold $\cM$  with an inner product $\langle \cdot,\cdot\rangle_x$ on its tangent spaces, which will be our subject of study henceforth.
%From now on we shall only denote it by $\cM$ without causing any confusion.
Let $f$ be a smooth function defined on $\cE \, ( =\mathbb{R}^n)$. Then $f$ is also a smooth function on $\cM$. If we denote $\Proj_{\cT_x\cM}$ to be the orthogonal projection onto the tangent space $\cT_x\cM$, then %similar to the Euclidean case,
one can define the \emph{Riemannian gradient}\/ and the \emph{Riemannian Hessian}\/ as follows:
\begin{definition}[Riemannian Gradient and Hessian]
	Let $f$ be a smooth function on $\cE$. The Riemannian gradient $\grad f(x)$ of $f$ with respect to a submanifold $\cM$ is a tangent vector in $\cT_x\cM$ defined by
	\begin{equation}
	\label{defn:Rie-grad}
	\grad f(x) = \Proj_{\cT_x\cM}(\nabla f(x)).
	\end{equation}
	For any $z \in\cT_x\cM$, the operation of the Riemannian Hessian $\Hess f(x)$ of $f$ on $z$ is defined as
	\begin{equation}
	\label{defn:Rie-Hessian}
	\Hess f(x)[z] = \Proj_{\cT_x\cM}(\mathrm{D}\grad f(x)[z]),
	\end{equation}
	where $\mathrm{D}\grad f(x)$ is the differential of $\grad f(x)$ in the usual sense. In terms of Jacobian matrix $\mathrm{J}\grad f(x)$, we have $\mathrm{D}\grad f(x)[z] = (\mathrm{J}\grad f(x))[z]$.
\end{definition}
A core ingredient in Riemannian optimization is the {\it retraction}\/ defined as follows (see \cite{Retraction} for more details).
%One can also read \cite{Retraction} for deeper understanding.
\begin{definition}[Retraction]\label{defn:retraction}
	Let $\Retr(x,\xi):\cT\cM\rightarrow\cM$ be a mapping from the tangent bundle $\cT\cM$ to the manifold $\cM$. Then we call $\Retr(\cdot,\cdot)$ a retraction if
	\begin{eqnarray}
	\label{defn:1st-retraction}
	\Retr(x,0) = x,\quad  \frac{d}{dt}\Retr(x,t\xi)\bigg|_{t=0}=\xi, \,\, \forall x\in\cM, \forall \xi\in\cT_x\cM.
	\end{eqnarray}
	We call $\Retr(\cdot,\cdot)$ a second-order retraction if it further satisfies
	\begin{equation}
	\label{defn:2nd-retraction}
	\frac{d^2}{dt^2}\Retr(x,t\xi)\bigg|_{t=0} \in \cT_x\cM^\perp, \,\, \forall x\in\cM, \forall \xi\in\cT_x\cM,
	\end{equation}
	where $\cT_x\cM^\perp$ denotes the orthogonal complement of $\cT_x\cM$ in $\cE$.
\end{definition}

Below is a useful property of retraction.

\begin{proposition}
	\label{prop:retraction-regularity}
	For a retraction $\Retr(\cdot,\cdot)$ on a compact submanifold $\cM\subset\cE$, there exist constants $L_1, L_2>0$ such that the following inequalities hold
	\begin{eqnarray}
	\|\Retr(x,\xi) - x \| & \leq & L_1\|\xi\|, \\
	\|\Retr(x,\xi) - x -\xi\| & \leq & L_2\|\xi\|^2, \label{L-2}
	\end{eqnarray}
	for all $x\in\cM$ and all $\xi\in\cT_x\cM$.
\end{proposition}

The constants $L_1$ and $L_2$ may depend on the manifold and the dimensions. As an example, the polar retraction and the QR retraction for the Stiefel manifold $\St_{n,r}$ all satisfy this regularity condition with some universsal constants $L_1$ and $L_2$ independent of $n$ and $r$; see \cite{Stief:SVRG-Jiang}. This Proposition was initially shown as a by-product of Lemma 3 in \cite{RM_glo}.

\subsection{The 1st and 2nd Order Optimality Conditions for Riemannian Optimization}

Consider the unconstrained problem \eqref{prob:Riemannian_Opt}. If $x^*$ is the optimal solution of the above problem, then \begin{equation}
\label{opt-cond-1st}
\grad f(x^*) = 0.
\end{equation}
Furthermore, if $f$ is second-order continuously differentiable, then
\begin{equation}
\label{opt-cond-2nd}
\langle \Hess f(x^*)[\xi],\xi \rangle \geq 0, \,\, \forall \xi\in\cT_{x^*}\cM
\end{equation}
is also satisfied. We refer the interested readers to  \cite{RieOpt:Yang-etal-2012,Opt_Manif:Absil-etal-2009} for more information on these optimality conditions. Consequently, %we can define the corresponding first and second-order $\epsilon$-stationary points by \eqref{def:1st-epsolu} and \eqref{def:2nd-epsolu} respectively.
%\begin{definition}
%	Consider the optimization problem \eqref{prob:Riemannian_Opt}, then
we call a point to be a first-order $\epsilon$-stationary point if
\be
\label{def:1st-epsolu}
\|\grad f(x^*)\|\leq \epsilon;
\ee
we call $x^*$ a second-order $\epsilon$-stationary point if, in addition, it also satisfies
\be
\label{def:2nd-epsolu}  \langle \Hess f(x^*)[\xi],\xi\rangle \geq -\sqrt{\epsilon}\|\xi\|^2 ,\,\, \forall \xi\in\cT_{x^*}\cM.
\ee
%\end{definition}

\subsection{The Extended Retraction}
To begin with, we first propose the \emph{extended retraction} and discuss its properties. As is defined in Definition \ref{defn:retraction}, a retraction is a mapping from the tangent bundle $\cT\cM$ to the manifold $\cM$ and is not defined outside of $\cT\cM$. Note that both $\cM$ and the tangent spaces $\cT_x\cM$ are embedded in $\cE$ and are parameterized with Euclidean coordinates, it will be convenient if we can use the usual differential operators in the Euclidean space. This requires the differentiability of the mapping in an open set of $\cE$ rather than restricted to $\cM$ or $\cT_x\cM$ with no interior. Hence we propose to work with the following \emph{extended retraction}, which extends $\Retr(x,\cdot)$ to the whole $\cE$ for all $x\in\cM$.

\begin{definition}[Extended Retraction] \label{def:extended-retraction}
	For a given
	retraction $\Retr(\cdot,\cdot):\cT\cM\rightarrow\cM$, and a given $x\in\cM$,  we can continuously and smoothly extend $\Retr(x,\cdot)$ from $\cT_x\cM$ to the whole space $\cE$ by defining the following extended retraction 
	\begin{equation}
	\label{lemma:ext-retr}
	\Retr(x,z) := \Retr(x, \Proj_{\cT_x\cM}(z)), \,\, \forall z\in\cE, \forall x\in\cM,
	\end{equation}
	which naturally extends $\Retr(x,\cdot)$ from $\cT_x\cM$ to the whole space $\cE$.
\end{definition}

%The extended retraction can also be written as 	$$\Retr(x,\cdot) = \Retr(x,\cdot)\circ \Proj_{\cT_x\cM}.$$
Note that for any $z\in\cT_x\cM$, the extended retraction remains the original retraction. Without loss of generality, we make the following assumption.

\begin{assumption}	\label{assumption:ext-retr}
	For the retraction $\Retr(\cdot,\cdot)$ under consideration, we assume that they are already extended smoothly to the whole space $\cE$ by incorporating \eqref{lemma:ext-retr}. Consequently, the following relationship holds
	$$\Retr(x,\cdot) = \Retr(x,\cdot)\circ \Proj_{\cT_x\cM}.$$
\end{assumption}

\begin{proposition}
	\label{prop:Jacobi-Retration}
	Suppose that for a submanifold $\cM\subset\cE$ the retraction $\Retr(\cdot,\cdot)$ satisfies Assumption \ref{assumption:ext-retr}. Then,
	\begin{equation}
	\label{assumption:ext-retr-1}
	\J_{\xi}\Retr(x,0) = \Proj_{\cT_x\cM},\,\, \forall x\in\cM.
	\end{equation}
\end{proposition}
\begin{proof}
	For any $x\in \cM, \forall \eta\in\cE$, Assumption \ref{assumption:ext-retr} and \eqref{defn:1st-retraction} guarantee that
	$$\J_{\xi}\Retr(x,0)[\eta] = \frac{d}{dt}\Retr(x,t\eta)\bigg|_{t=0} = \frac{d}{dt}\Retr(x,t\Proj_{\cT_x\cM}(\eta))\bigg|_{t=0} = \Proj_{\cT_x\cM}(\eta).$$
	The result thus follows.
\end{proof}

\subsection{The Pullback and Its Properties}
For any smooth function $f$ on $\cM$ and a retraction $\Retr(\cdot,\cdot)$, the pullback of $f$ at point $x$, denoted by $\hat{f}_x$, is defined as
\begin{equation} \label{f-hat}
\hat{f}_x(\xi) = f(\Retr(x,\xi)), \forall \xi\in\cT_x\cM.
\end{equation}
It locally reparametrizes a function with the points on a subspace $\cT_x\cM\subset\cE$ instead of the points on the manifold $\cM$. When differentiation is performed, the pullback is automatically extended to the whole space $\cE$ through the extended retractions. The gradient and Hessian of the pullbacks connect to the Riemannian gradient and the Riemannian Hessian through the relationships shown in the next three propositions.
\begin{proposition}[Pullback Gradient]
	\label{prop:pullback-grad}
	Under Assumption \ref{assumption:ext-retr}, the pullback $\hat{f}_x$ satisfies
	\begin{equation}
	\label{prop:pullback-grad-1}
	\nabla_{\xi} \hat{f}_x(0) = \grad f(x), \, \forall x\in\cM.
	\end{equation}
\end{proposition}
\begin{proof} By Proposition \ref{prop:Jacobi-Retration},
	$\nabla_{\xi} \hat{f}_x(0) = \nabla_{\xi} f(\Retr(x,0)) = \J_\xi \Retr(x,0)\left[\nabla f(x)\right] = \grad f(x).$
\end{proof}
The gradient of the pullback is equal to the Riemannian gradient. However, the Hessian of the pullback is not necessarily equal to the Riemannian Hessian.
\begin{proposition}[Pullback Hessian \cite{RM_glo}]
	\label{prop:pullback-Hess}
	%	Let $f$ and $\Retr(\cdot,\cdot)$ satisfy the same set of conditions of Proposition \ref{prop:pullback-grad}.
	Under Assumption \ref{assumption:ext-retr} and \eqref{f-hat},
	it holds that
	\begin{equation}
	\label{prop:pullback-Hess-1}
	\big\langle \nabla_{\xi}^2\hat{f}_x(0)(x)[\xi], \xi\big\rangle  = \big\langle \Hess f(x)[\xi], \xi \big\rangle	+ \big\langle \grad f(x), \frac{d^2}{dt^2}\Retr(x,t\xi)\big|_{t=0}\big\rangle, \, \forall \xi\in\cT_x\cM.
	\end{equation}
\end{proposition}
The proof of this proposition can be found in Appendix C of \cite{RM_glo}. %However, the proof in \cite{RM_glo} assumes in-depth knowledge of Riemannian geometry. We provide here an alternative proof using only elementary calculus in the appendix, for the benefit of interested readers.

\begin{corollary}
	\label{corollary:pullback-Hess}
	If $\Retr(\cdot,\cdot)$ is a second-order retraction, then the pullback Hessian at $0$ coincides with the Riemannian Hessian on the tangent space, i.e.,
	$$\big\langle \nabla_{\xi}^2\hat{f}_x(0)(x)[\xi], \xi\big\rangle  = \big\langle \Hess f(x)[\xi], \xi \big\rangle, \, \forall \xi\in\cT_x\cM.$$
\end{corollary}
\begin{proof}
	By \eqref{defn:2nd-retraction}, we have
	$\frac{d^2}{dt^2}\Retr(x,t\xi)\big|_{t=0}\in\cT_x\cM^\perp.$
	Therefore, we have $$\big\langle \grad f(x),  \frac{d^2}{dt^2}\Retr(x,t\xi)\big|_{t=0}\big\rangle=0.$$
	In combination with Proposition \ref{prop:pullback-Hess}, this proves the corollary.
\end{proof}

%By the regularity conditions guaranteed by Proposition \ref{prop:retraction-regularity}, another corollary can be directly proved.
\begin{corollary}
	\label{corollary:pullback-Hess-Grad-bound}
	%Let $f$ and $\Retr(\cdot,\cdot)$ be as defined in Proposition \ref{prop:pullback-Hess}.
	Under Assumption \ref{assumption:ext-retr} and \eqref{f-hat},
	and suppose that $\Retr(\cdot,\cdot)$ satisfies Proposition \ref{prop:retraction-regularity} with parameter $L_2$ introduced in \eqref{L-2}. Then
	\begin{equation}
	\big|\big\langle(\nabla_{\xi}^2\hat{f}_x(0)-\Hess f(x))[\eta],\eta\big\rangle\big| \leq 2L_2\|\grad f(x)\|\|\eta\|^2, \,  \forall \eta\in\cT_x\cM,
	\forall x\in\cM.
	\end{equation}
\end{corollary}
\begin{proof}
	For any fixed $x\in\cM$ and $\eta\in\cT_x\cM$, let us denote $Y(t) = \Retr(x,t\eta).$
	Consequently $Y(0) = x$ and  $Y'(0) = \eta$. By the continuity of the norm,
	\begin{eqnarray*}
		\|Y''(0)\| & = & \left\|\lim_{t\rightarrow0}\frac{Y(t)-Y(0)-Y'(0)t}{\half t^2}\right\|  = \lim_{t\rightarrow0}\frac{\big\|Y(t)-Y(0)-Y'(0)t\big\|}{\half t^2} \\  &\leq & \lim_{t\rightarrow0}\frac{L_2t^2\|\eta\big\|^2}{\half t^2} = 2L_2\|\eta\|^2 .
	\end{eqnarray*}%	That is $\|Y''(0)\| \leq L_2\|\eta\|^2$.
	On the other hand, by Proposition \ref{prop:pullback-Hess},
	$$\big|\big\langle(\nabla_{\xi}^2\hat{f}_x(0)-\Hess f(x))[\eta],\eta\big\rangle\big| \leq \|\grad f(x)\| \cdot \|Y''(0)\|.$$
	Combining this with the bound on $\|Y''(0)\|$ yields the desired result.
\end{proof}

For the iteration complexity of the Riemannian gradient descent or the Riemannian trust-region methods, it is sufficient to know the gradient of the pullback at the origin, i.e., $\nabla_{\xi}\hat{f}_x(0)$. However, to derive faster local convergence we will also need to analyze the pullback gradient in a neighbourhood of $0$.
\begin{proposition}[Pullback Gradient in a Neighbourhood]
	\label{prop:pullback-grad-Neighbourhood}
	%	Let $f$ and $\Retr(\cdot,\cdot)$ both satisfy the previous assumptions.
	Under Assumption \ref{assumption:ext-retr} and \eqref{f-hat},
	for any $x\in\cM$ and for any $\xi\in\cT_x\cM$ with $\|\xi\|$ sufficiently small, we have
	\begin{equation}
	\label{prop:pullback-grad-Neighbourhood-1}
	\|\grad f(y)\|\leq \frac{1}{1-\|\Proj_{\cT_y\cM}-\J_\xi \Retr(x,\xi)\|_2}\|\nabla_{\xi}\hat{f}_x(\xi)\|,
	\end{equation}
	where $y = \Retr(x,\xi)$, and $\|\cdot\|_2$ denotes the matrix spectral norm.
\end{proposition}

Before proving this proposition, let us consider an example.
%Let us see the corresponding case on the sphere with
Consider the retraction to the unit sphere. In that case, the retraction is $\Retr(x,\xi) := \frac{x+ \xi }{\|x+ \xi\|}$. By direct calculation, we have $\nabla_{\xi}\hat{f}_x(\xi) = \frac{1}{\|x+\xi\|}(I-yy^\top)\nabla f(y) = \frac{1}{\|x+\xi\|}\grad f(y)$. Therefore, if $\|\xi\|$ is small enough, then the difference between $\nabla_{\xi}\hat{f}_x(\xi)$ and $\grad f(y)$ can be controlled.
\begin{proof}
	Suppose $\cE = \R^n$, $\mbox{\rm dim}(\cM) = d<n$. For any $x\in\cM$ and any $\xi\in\cT_x\cM$, let $\phi:\cE\rightarrow\R^{n-d}$ be a smooth local equation of $\cM$ around point $y:= \Retr(x,\xi)$ (see e.g.\ \cite{Retraction}). That is, there exists a local neighbourhood $U_y$ of $y$ in $\cE$, satisfying
	$z\in\cM\cap U_y \Longleftrightarrow \phi(z) = 0, z\in U_y.$ For special examples, such as the unit sphere and Stiefel manifold, this local equation is actually global. Therefore, $\exists\delta>0 \mbox{ }$ such that $\forall \eta\in\cE, \|\eta\|^2 = 1$, we have $\phi(\Retr(x,\xi+t\eta))=0, \forall t\in (-\delta,\delta)$. Hence
	$$0 = \frac{d}{dt}\phi(\Retr(x,\xi+t\eta))\big|_{t=0} = \J\phi(y)\J_{\xi}\Retr(x,\xi)\eta, \,\, \forall \|\eta\|^2 = 1.$$
	Consequently, $\J\phi(y)\J_{\xi}\Retr(x,\xi) = 0$, which means that
	\begin{equation}
	\mathrm{Ran}(\J_{\xi}\Retr(x,\xi)) \subset \mathrm{Ran}(\J\phi(y)^\top)^{\perp} = \mathrm{Ran}(\Proj_{\cT_y\cM}). \nonumber
	\end{equation}
	Therefore
	$\Proj_{\cT_y\cM}\J_{\xi}\Retr(x,\xi) = \J_{\xi}\Retr(x,\xi).$
	By direct calculation we have
	\begin{eqnarray*}
		\grad f(y) - \nabla_{\xi}\hat{f}_x(\xi)& = & \Proj_{\cT_y\cM}\nabla f(y) - \J_\xi \Retr(x,\xi)^\top\nabla f(y) \\
		& = & \Proj_{\cT_y\cM}^2\nabla f(y) - \J_\xi \Retr(x,\xi)^\top\Proj_{\cT_y\cM}\nabla f(y) \\
		& = & \left(\Proj_{\cT_y\cM} - \J_\xi \Retr(x,\xi)^\top \right) \grad f(y) .
	\end{eqnarray*}
	Since $y = x$ when $\xi = 0$ and $\J_\xi \Retr(x,\xi)^\top = \Proj_{\cT_x\cM}$, it follows that
	$$\|\Proj_{\cT_y\cM} - \J_\xi \Retr(x,\xi)^\top\|_2<1$$
	%will hold for all enough small $\xi$'s.
	when $\xi$ is sufficiently small.
	Therefore,
	$$\|\grad f(y)\|-\|\nabla_{\xi}\hat{f}_x(\xi)\|\leq \|\Proj_{\cT_y\cM} - \J_\xi \Retr(x,\xi)^\top\|_2\|\grad f(y)\|.$$
	The proposition follows by rearranging the terms in the above inequality.
\end{proof}
If $\cM$ is compact, then we have the following:
\begin{corollary}[Proposition \ref{prop:pullback-grad-Neighbourhood} Refined]
	\label{corollary:pullback-grad-Neighbourhood}
	%	Let $f$ and $\Retr$ satisfy the conditions in Proposition \ref{prop:pullback-grad-Neighbourhood},
	Under Assumption \ref{assumption:ext-retr} and \eqref{f-hat},
	supposing that $\cM$ is compact, then there exists a constant $C_g>0$ such that for all $x\in \cM$ and for all $\xi\in\cT_x\cM$ with $\|\xi\|_F\leq C_g$, we have
	\begin{equation}
	\label{theorem:pullback-grad-Neighbourhood-1}
	\|\grad f(\Retr(x,\xi))\|_F\leq 2\|\nabla_{\xi}\hat{f}_x(\xi)\|.
	\end{equation}
	%If the $Retr(\cdot,\cdot)$ is chosen to the QR retraction. Then, for $\forall x\in St_{n,r}$, for $\forall\xi\in\cT_xSt_{n,r}$, we only need to choose $\|\xi\|_F\leq\frac{1}{6+\sqrt{2}+\sqrt{10}}$ to make \eqref{theorem:pullback-grad-Neighbourhood-1} hold again.
\end{corollary}
The constant $C_g$ depends on the manifold $\cM$ and the retraction. For instance,
$C_g = 1/8.62$ if $\cM$ is the Stiefel manifold with polar retraction.

Next we consider the Lipschitz continuity of the pullback Hessian at the origin. That is, we wish to establish for any $x\in \cM$ an inequality in the form of
\begin{equation}
\label{pullback-Lips}
\big|\big\langle(\nabla_{\xi}^2\hat{f}_x(\eta) - \nabla_{\xi}^2\hat{f}_x(0))[\nu],\nu\big\rangle\big| \leq L_H\|\eta\|, \,\,\, \forall \eta\in\cT_x\cM, \ \forall \nu\in\cT_x\cM, \|\nu\|=1,
\end{equation}
where $L_H$ is independent of $x$. Such property is non-trivial because the tangent bundle is noncompact (unbounded).
However, it is true in the case of the Stiefel manifold with polar retraction, and we shall explicitly compute this constant in later sections.
%However, it is not always easy to guarantee the existence of such constant due to the non-compactness of the tangent bundle.
Fortunately, we only need a weaker form of local Lipschitz continuity,
%shall still keep the algorithm working and is guaranteed to exist.
which is true and is sufficient for our analysis of the general case:
\begin{lemma}[Local Lipschitz Continuity of Pullback Hessian]
	\label{lemma:Pullback-Lips-Loc}
	%If the extended retraction $\Retr(\cdot,\cdot)$ is smooth, then f
	For any $R>0$, there exists a constant $L_H^R>0$ such that condition \eqref{pullback-Lips} holds for $\|\eta\| \leq R$:
	\begin{equation}
	\label{pullback-Loc-Lips}
	\big|\big\langle(\nabla_{\xi}^2\hat{f}_x(\eta) - \nabla_{\xi}^2\hat{f}_x(0))[\nu],\nu\big\rangle\big| \leq L_H^R\|\eta\|, \,\, \forall \eta\in\cT_x\cM, \|\eta\|\leq R,  \forall \nu\in\cT_x\cM, \|\nu\|=1.
	\end{equation}
\end{lemma}\vspace{-0.15cm}
This lemma is a direct consequence of the Lipschitz continuity of $\nabla^2f$, the smoothness of $\Retr(\cdot,\cdot)$ and the compactness of $\{\eta:\:\eta\in\cT_x\cM, x\in\cM, \|\eta\|\leq R \}$.
Finally, we present a local Lipschitz property on the Riemannian Hessian in the next lemma, whose proof is in Appendix~\ref{appendix-B}.
%Though not relevant to the pullbacks, we still present it here due to its connection to the local Lipschitz continuity of pullback Hessian.
\begin{lemma}[Local Lipschitz Continuity of Riemannian Hessian]
	\label{lemma:Lip-Rie-Hess}
	Let $\cM\subset\cE$ be a compact submanifold and let $f$ be a smooth function with Lipschitz continuous Hessian. Then there exist constants $d_{\cM}, D > 0$ such that for any $x,y\in\cM$, if $\|x-y\| \leq d_{\cM}$ then
	\begin{equation}
	\label{lemma:Lip-Rie-Hess-1}
	|\lambda_{\min}^\cM(\Hess f(x)) - \lambda_{\min}^\cM(\Hess f(y))| \leq D\|x-y\|
	%\bigg|inf_{\xi\in\cT_x\cM}\bigg\{\frac{\langle \Hess f(x)[\xi],\xi\rangle}{\|\xi\|^2}\bigg\}-
	%inf_{\eta\in\cT_y\cM}\bigg\{\frac{\langle \Hess f(y)[\eta],\eta\rangle}{\|\eta\|^2}\bigg\}\bigg|\leq D\|x-y\|,
	\end{equation}
	where $\lambda_{\min}^\cM(\Hess f(x)):= \inf_{\xi\in\cT_x\cM}\bigg\{\frac{\langle \Hess f(x)[\xi],\xi\rangle}{\|\xi\|^2}\bigg\}$ and $\|\cdot\|$ is the Euclidean norm.
\end{lemma}

\section{Cubic Regularized Newton's Method} \label{sec:algo}

\subsection{The Basic Algorithm and Its Convergence}

%To start with, let us recall the local Lipschitz property \eqref{pullback-Loc-Lips} and expand on this. %We have: %Through standard proofs one can get
We shall now expand on the local Lipschitz property \eqref{pullback-Loc-Lips}. The following are some estimations of the residuals.

\begin{proposition}
	\label{prop:3rd-descent-lemma}
	Let $\Retr(\cdot,\cdot)$ be a second-order retraction on $\cM$. Suppose condition \eqref{pullback-Loc-Lips} holds for all $x\in\cM$ with a uniform constant $L_H^R$ for the pullbacks $\hat{f}_x$. Then
	\begin{equation}
	\label{prop:Lip-1}
	\big\|\nabla_{\xi}\hat{f}_x(\xi) - \nabla_{\xi}\hat{f}_x(0) - \nabla^2_{\xi}\hat{f}_x(0)[\xi]  \big\|\leq \frac{L_H^R}{2}\|\xi\|^2,
	\end{equation}
	\begin{equation}
	\label{prop:Boumal-Lip}
	\big|\hat{f}_x(\xi) - f(x) - \langle \grad f(x),\xi \rangle - \frac{1}{2}\big\langle \Hess f(x)[\xi], \xi \big\rangle  \big|\leq \frac{L_H^R}{6}\|\xi\|^3
	\end{equation}
	for all $\xi\in\cT_\cM, \|\xi\|\leq R,$ where the constant $L_H^R$ is independent of $x$.
\end{proposition}
%\begin{remark}
We remark here that estimate \eqref{prop:Boumal-Lip} first appeared in \cite{RM_glo} without condition \eqref{pullback-Loc-Lips} and the constraint that $\|\xi\|\leq R$. As a consequence, it now follows from \eqref{pullback-Lips}. %The converse derivation is not true though. %However, it is not sufficient to indicate \eqref{pullback-Lips}.
%\end{remark}
This proposition immediately suggests that for any $\sigma > L_H^R$,
\begin{equation}
\label{fun-upbound}
m_{x,\sigma}(\xi) := \hat{f}_x(0) + \langle \grad f(x),\xi \rangle + \frac{1}{2}\langle \Hess f(x)[\xi],\xi \rangle + \frac{\sigma}{6}\|\xi\|^3
\end{equation}
is an upper bound of $\hat{f}_x(\xi)$ in the subspace $\cT_x\cM$ if $\|\xi\|\leq R$. Therefore, whenever $R$ is large enough so as to ensure the global minimum of $m_{x,\sigma}(\xi)$ is in the interior of the disk $\xi\in\cT_\cM, \|\xi\|\leq R$, then following the principle of majorization-minimization it makes sense to minimize %perform unconstrained minimization for
$m_{x,\sigma}(\xi)$ over $\cT_x\cM$ as an iterative subroutine, instead of minimizing $\hat{f}_x(\xi)$ itself. %in  subproblems.
More specifically, the iterates run as:
\begin{equation}
\begin{cases}
\xi_{k} := \arg\min m_{x_k,\sigma}(\xi), \mbox{ subject to } \xi\in\cT_x\cM, \\
x_{k+1} := \Retr(x_k,\xi_k).
\end{cases}
\end{equation}
Note that if we denote the projection onto $\cT_x\cM$ to be operator $P_x$, then the constrained cubic subproblem can be equivalently rewritten as
\begin{eqnarray}
\label{prob:cubic-sub}
\xi_{k} & = & \arg\min_\xi \hat{m}_{x_k,\sigma}(\xi) \\ &:= & \arg\min_\xi \hat{f}_x(0) + \langle \grad f(x),\xi \rangle + \frac{1}{2}\langle P_x\circ \Hess f(x)\circ P_x[\xi],\xi \rangle + \frac{\sigma}{6}\|\xi\|^3.
\end{eqnarray}
In principle, we should also replace $\grad f(x)$ by $P_x[\grad f(x)]$. Since $\grad f(x)$ is in $\cT_x\cM$, the projection becomes redundant.
This means that we essentially end up with an unconstrained problem over $\cE$, which can be solved to global optimality; see e.g.\cite{Cubic:Nestrov,Cubic:Carmon}.

%Finally, we introduce a set of %denote the a bunch of constants as follows.
Below we present a number of constants to be used later.
Let $\nabla^2f$ be Lipschitz continuous on the convex hull $\mathrm{Conv}(\cM)$ with Lipschitz constant $\ell_H$. Define 	$\ell_f := \max_{x\in\mathrm{Conv}(\cM)}\|\nabla^2f(x)\|_2$
as the Lipschitz constant for $\nabla f$ over $\mathrm{Conv}(\cM)$, denote $G := \max_{x\in\cM}\|\nabla f(x)\|_F$, and $k_B := \max_{x\in\cM}\max_{\xi\in\cT_x\cM,\|\xi\|=1}\|\Hess f(x)[\xi]\|.$

\begin{algorithm2e}[H]
	\caption{Cubic Regularized Newton's Method over Riemannian Manifold}
	\label{algo:Cubic}%\LinesNumberedHidden
	{\it Input:} an initial point $x_0\in\cM$, a retraction $\Retr(\cdot,\cdot)$, a parameter $\sigma>\max\left\{\left(\sqrt{10L_2k_B + \frac{2}{3}L_H^R+9L_2^2G}+3L_2\sqrt{G}\right)^2,1\right\}$ where $L^R_H$ is defined in Lemma \ref{lemma:Pullback-Lips-Loc} with $R = 3k_B + 3\sqrt{G}$
	%a lower bound $\underline{f}$ of the objective function, the parameters  $L_2$, $\ell_f$, $k_B$, and
	%a stepsize  parameter $\alpha = \frac{1}{2L_2G + L^2_1\ell_f}$ 
	and an iteration number $T$. \\
	%Set $k,t\leftarrow 0,$ and $\mathcal{I} = \emptyset.$\\
	\For{$k = 0,2,...,T-1$ }{
		%\If{$\|\grad f(x_k)\| \geq \frac{k_B^2}{\sigma}$}{
		%	Update $x_{k+1} = \Retr(x_k,-\alpha \cdot \grad f(x_k))$ and $k\leftarrow k+1.$}
		Solve $\xi_{k} = \arg\min_\xi \hat{m}_{x_k,\sigma}(\xi).$ \\
		Update $x_{k+1} = \Retr(x_k,\xi_k).$\\
		%Add $k$ to $\mathcal{I}$ and set  $k\leftarrow k+1$, $t\leftarrow t+1$.		
	}
	{\it Output:} Let $k^* := \arg\min_{0\leq k\leq T-1}\|\xi_{k}\|^3,$ and return $x_{k^*+1}$.
\end{algorithm2e}
%One remark is that this algorithm works for all retractions. However, if a second-order retraction is available, then we only need $\sigma>\max\{\frac{2}{3}L_H^R,1\}$.

%The algorithm stipulates that if the norm of Riemannian gradient $\|\grad f(x_k)\|$ is greater than some threshold value $\frac{k_B^2}{\sigma}$, then a gradient step is taken; otherwise, a cubic regularized Newton step is taken. Because a gradient step amounts to a minimum function value reduction, the gradient steps will only be taken finitely many times. Therefore the gradient steps does not affect the rate of convergence at all. It is worth mentioning that this is similar to the Riemannian trust-region (RTR) method proposed in \cite{RM_glo,RM_tru,RM_tru1}, which also needs to take some gradient steps (Cauchy steps) before the Riemannian gradient is small enough. However, the threshold for the gradient steps in these papers are of the order $\mathcal{O}(\epsilon)$ while the threshold here is of the order $\mathcal{O}(1)$.
% We shall also comment here that if the Lipschitz constant exists globally for the pullback Hessian, then there is no need to take gradient steps. This property is also reflected in our numerical experiments on the Stiefel manifold.
\begin{theorem}
	\label{theorem:Manifold-complexity} Let the sequence ${(X_k,\xi_k)}$ be generated by Algorithm \ref{algo:Cubic}, with the parameters chosen to satisfy
	$\sigma > \max\left\{\left(\sqrt{10L_2k_B + \frac{2}{3}L_H^R+9L_2^2G}+3L_2\sqrt{G} \right)^2,1\right\}$ and $R = 3k_B + 3\sqrt{G}$, if  we choose to set  $k^* := \arg\min_{0\leq k\leq T-1}\|\xi_{k}\|^3$ and 
	$$T \geq \frac{4(f(x_0)-\underline{f})}{\tau_1} \cdot\max \left\{\frac{1}{C_g^3},\left(\frac{L_1}{d_\cM}\right)^{3},\frac{\left(DL_1+\sigma/2\right)^{\frac{3}{2}}}{\epsilon^{\frac{3}{2}}}, \frac{\tau_2^{\frac{3}{2}}}{\epsilon^{\frac{3}{2}}}\right\}, $$
	where $C_g$ is defined in Corollary \ref{corollary:pullback-grad-Neighbourhood}, $d_\cM$ is defined in Lemma \ref{lemma:Lip-Rie-Hess}, $\underline{f}$ is any lower bound of the optimal value, $\tau_1 = \sigma - 10L_2k_B - 6L_2\sqrt{\sigma G} - \frac{2}{3}L_H^R$ and $\tau_2 = \sigma+L_H^R+10L_2k_B+6L_2\sqrt{\sigma G}$. Then, the returned point $x_{k^*+1}$ is a second-order $\epsilon$-stationary point satisfying \eqref{def:1st-epsolu} and \eqref{def:2nd-epsolu}.
\end{theorem}
\begin{proof}
	%	First, let us show the guaranteed constant descent in the gradient steps. Define $L_f = 2L_2G+L_1^2\ell_f.$
	%	In this case, the updating rule will be  $x_{k+1} = \Retr(x_k,-\alpha\cdot \grad f(x_k)).$  By the Lipschitz continuity of $\nabla f$,
	%	\begin{eqnarray*}
	%		f(x_{k+1}) & \leq & f(x_k) + \langle \nabla f(x_k), x_{k+1}-x_k\rangle + \frac{\ell_f}{2}\|x_{k+1}-x_k\|^2 \\
	%		& \leq & f(x_k) + \langle \nabla f(x_k), -\alpha\cdot \grad f(x_k)\rangle  + \frac{\ell_f}{2}\|x_{k+1}-x_k\|^2 \\
	%		& & + \|\nabla f(x_k)\|\|x_{k+1} - x_k + \alpha\cdot \grad f(x_k)\|.
	%	\end{eqnarray*}
	%	Note that $\langle \nabla f(x_k), -\alpha\cdot \grad f(x_k)\rangle = -\alpha\|\grad f(x_k)\|^2$. According to the regularity conditions in Proposition \ref{prop:retraction-regularity}, $\|x_{k+1}-x_k\| \leq L_1\|\alpha\cdot \grad f(x_k)\| \mbox{ and }\|x_{k+1} - x_k + \alpha\cdot \grad f(x_k)\| \leq L_2\|\alpha\cdot \grad f(x_k)\|^2.$
	%	Consequently, we have
	%	\begin{eqnarray*}
	%		f(x_{k+1})& \leq & f(x_k) - \left( \alpha - \alpha^2GL_2-\alpha^2\frac{L_1^2\ell_f}{2}\right)\|\grad f(x_k)\|^2 = f(x_k) - \frac{1}{2L_f}\|\grad f(x_k)\|^2.
	%	\end{eqnarray*}
	%	Together with the condition that $\|\grad f(x_k)\|  \geq \frac{k_B^2}{\sigma}$, we have
	%	$$f(x_k) - f(x_{k+1}) \geq \frac{k_B^4}{2L_f\sigma^2}.$$
	%	This indicates that as long as $\|\grad f(x_k)\|  \geq \frac{k_B^2}{\sigma}$, a constant amount of decrease will be guaranteed. Therefore, at most finitely many number of gradient steps will take place.
	%	Second, we turn to the convergence behavior of the cubic regularized Riemannian Newton steps. 
	For the ease of notation, denote $g_k:= \grad f(x_k) \mbox{ and } B_k = P_{x_k}\Hess f(x_k)P_{x_k},$ where $P_{x_k}$ is the orthogonal projection onto $\cT_{x_{k}}\cM$.   
	Then the subproblem can be represented as
	$	\xi_{k} =  \arg\min_\xi g_k^\top\xi + \frac{1}{2}\xi^\top B_k\xi + \frac{\sigma}{6}\|\xi\|^3.$	According to Lemma 2.2 in \cite{Cubic:Cartis-1}, $\|\xi_{k}\|\leq \frac{3}{\sigma}\max\left\{k_B,\sqrt{\sigma\|g_k\|}\right\}. $
	Since $\|g_k\|\leq \|\nabla f(x_k)\|\leq G,\sigma\geq1$, we have
	\begin{equation}
	\label{thm:Gen-convgce-2}
	\|\xi_{k}\|\leq \frac{3}{\sigma}\max\{k_B,\sqrt{\sigma G}\}\leq   \frac{3k_B}{\sigma} + 3\sqrt{\frac{G}{\sigma}}\leq R.
	\end{equation}
	This validate our choice of $R$. By the global optimality conditions provided in \cite{Cubic:Nestrov}
	\begin{equation}
	\label{thm:convgce-1}
	(B_k+\lambda^*_kI)\xi_k + g_k = 0,\quad
	\lambda^*_k = \frac{\sigma}{2}\|\xi_{k}\|,\quad
	B_k+\lambda^*_kI \succeq 0.
	\end{equation} 
	The first two conditions of \eqref{thm:convgce-1} further result in
	\begin{equation}
	\label{thm:convgce-2}
	g_k^\top\xi_{k}+\xi_{k}^\top B_k\xi_k+\frac{\sigma}{2}\|\xi_k\|^3 = 0.
	\end{equation}

	In the absence of second-order retraction, $\Hess f(x_k)$ is no longer equal to the Hessian of the pullback $\nabla_{\xi}^2\hat{f}_{x_k}(0)$ and the majorization property
	$\hat{m}_{x_k,\sigma}(\xi) \geq \hat{f}_{x_k}(\xi)$
	no longer holds. In this case, let us denote the matrix $H_k: = P_{x_k}\nabla_{\xi}^2\hat{f}_{x_k}(0)P_{x_k}.$
	%	As we shall show later, 
	%Let us assume that $\|\xi_{k}\|\leq R$ for the moment, which will be shown to hold later.
	We start from  Lemma \ref{lemma:Pullback-Lips-Loc} with constant $L_H^R$ and get
	\begin{eqnarray*}
		\hat{f}_{x_k}(\xi_{k}) & \leq & \hat{f}_{x_k}(0) + g_k^\top\xi_{k} + \frac{1}{2}\xi_{k}^\top H_k\xi_{k}+ \frac{L_H^R}{6}\|\xi_{k}\|^3 \\
		& = & \hat{f}_{x_k}(0) -\xi_{k}^\top B_k\xi_{k} + \frac{1}{2}\xi_{k}^\top H_k\xi_{k}+ \left(\frac{L_H^R}{6}-\frac{\sigma}{2}\right)\|\xi_{k}\|^3 \\
		& = & \hat{f}_{x_k}(0) -\half\xi_{k}^\top \left(B_k + \frac{\sigma}{2}\|\xi_{k}\|I\right)\xi_{k} + \frac{1}{2}\xi_{k}^\top(H_k-B_k)\xi_{k}- \left(\frac{\sigma}{4}-\frac{L_H^R}{6}\right)\|\xi_{k}\|^3.
	\end{eqnarray*}
	Note that by condition \eqref{thm:convgce-1} we have  $\half\xi_{k}^\top(B_k + \frac{\sigma}{2}\|\xi_{k}\|I)\xi_{k}\geq 0$,
	and by Corollary \ref{corollary:pullback-Hess-Grad-bound} we have
	$\frac{1}{2}\xi_{k}^\top(H_k-B_k)\xi_{k} \leq  L_2\|g_k\|\|\xi_{k}\|^2.$
	Combining these inequalities leads to
	\begin{equation}
	\label{thm:Gen-convgce-1}
	f(x_k) - f(x_{k+1})\geq \left(\frac{\sigma}{4}-\frac{L_H^R}{6} -  L_2 \cdot\frac{\|g_k\|}{\|\xi_{k}\|}\right)\|\xi_{k}\|^3.
	\end{equation} 
	Recall that the optimality condition in \eqref{thm:convgce-1} gives $-g_k = (B_k+\frac{\sigma}{2}\|\xi_{k}\|I)\xi_{k}.$
	By combining this equality with \eqref{thm:Gen-convgce-2}, we have
	\begin{equation}
	\label{thm:Gen-convgce-3}
	\|g_k\|\leq \left(\|B_k\|_2 + \frac{\sigma}{2}\|\xi_{k}\|\right)\|\xi_{k}\|\leq \left(\frac{5}{2}k_B+\frac{3}{2}\sqrt{\sigma G}\right)\|\xi_{k}\|.
	\end{equation}
	Putting \eqref{thm:Gen-convgce-1} and \eqref{thm:Gen-convgce-3} together yields
	\begin{equation}
	\label{thm:Gen-convgce-4}
	f(x_k) - f(x_{k+1})\geq \frac{1}{4}\left(\sigma - 10L_2k_B - 6L_2\sqrt{\sigma G} - \frac{2}{3}L_H^R\right)\|\xi_{k}\|^3 = \frac{\tau_1}{4}\|\xi_k\|^3.
	\end{equation}
	When $\sigma > \left(\sqrt{10L_2k_B + \frac{2}{3}L_H^R+9L_2^2G}+3L_2\sqrt{G} \right)^2$, the decrease is positive. Summing the inequalities up yields
	%$$\sum_{k\in\mathcal{I}}\|\xi_{k}\|^3 \leq \frac{4(f(x_0) - \underline{f})}{\sigma - 10L_2k_B - 6L_2\sqrt{\sigma G} - \frac{2}{3}L_H^R},$$
	$$\sum_{k\in\mathcal{I}}\|\xi_{k}\|^3 \leq \frac{4(f(x_0) - \underline{f})}{\tau_1},$$
	where the existence of $\underline{f}$ is guaranteed by the compactness of $\cM$. Following the way that $T$ and $k^*$ are set, we have
	\begin{equation}
	\label{thm:Gen-convgce-5}
	\|\xi_{k^*}\| \leq \min\left\{C_g,d_\cM/L_1,\tau_2^{-\half}\epsilon^\half,(DL_1+\sigma/2)^{-\half}\epsilon^\half\right\}.
	\end{equation}
	It remains to prove that it is an $\epsilon$-solution. To this end, note that
	$$\begin{cases}
	\|\nabla_{\xi}\hat{f}_{x_{k^*}}(\xi_{k^*}) - g_{k^*} - H_{k^*}\xi_{k^*}\| \leq \frac{L_H^R}{2}\|\xi_{k^*}\|^2, \\
	\|g_{k^*} + B_{k^*}\xi_{k^*}\| =  \frac{\sigma}{2}\|\xi_{k^*}\|^2,\\
	\|(H_{k^*} - B_{k^*})\xi_{k^*}\|\leq 2L_2\|g_{k^*}\|\|\xi_{k^*}\| \leq L_2(5k_B+3\sqrt{\sigma G})\|\xi_{k^*}\|^2,
	\end{cases}$$
	where the first inequality is due to Lemma \ref{lemma:Pullback-Lips-Loc} and  Proposition \ref{prop:3rd-descent-lemma}, the equality is due to condition \eqref{thm:convgce-1}, and the last inequality is due to Corollary \ref{corollary:pullback-Hess-Grad-bound} and
	\eqref{thm:Gen-convgce-3}. Combining these, we have %By these relationships, we get
	%$$\|\nabla_{\xi}\hat{f}_{x_{k^*}}(\xi_{k^*})\| \leq \frac{\sigma+L_H^R+10L_2k_B+6L_2\sqrt{\sigma G}}{2}\|\xi_{k^*}\|^2.$$
	$$\|\nabla_{\xi}\hat{f}_{x_{k^*}}(\xi_{k^*})\| \leq \frac{\tau_2}{2}\|\xi_{k^*}\|^2.$$
	Since $\|\xi_{k^*}\|\leq C_g$, we further obtain
	\begin{equation}
	\label{thm:Gen-convgce-6}
	\|\grad f(x_{k^*+1})\|\leq 2\|\nabla_{\xi}\hat{f}_{x_{k^*}}(\xi_{k^*})\| \leq \tau_2\|\xi_{k^*}\|^2\leq \epsilon.
	\end{equation}
	By condition \eqref{thm:convgce-1} and Lemma \ref{lemma:Lip-Rie-Hess},
	$$\lambda_{\min}(B_{k^*+1})\geq\lambda_{\min}(B_{k^*}) - DL_1\|\xi_{k^*}\|\geq -\left(DL_1+\frac{\sigma}{2}\right)\|\xi_{k^*}\| \geq -\sqrt{\epsilon}.$$
	Hence $\big\langle \Hess f(x_{k^*+1})[\eta], \eta\big\rangle \geq   -\sqrt{\epsilon}\|\eta\|^2, \,\, \forall \eta\in\cT_{x_{k^*+1}}\cM.$
	The proof is complete.
\end{proof}

\subsection{Speeding Up Local Convergence}
In this subsection, we investigate the possibility of speeding up the theoretical convergence rate under some additional conditions. One such condition is the so-called gradient-dominant property.

\subsubsection{Gradient-Dominant Functions}

\begin{definition}[Locally Gradient-Dominant Function]
	\label{defn:Grad-Domnt-func}
	For a smooth function $f$ defined on a manifold $\cM$, if for any local minimum point $\bar{x}$, there exists a neighbourhood $U_{\bar{x}}$ of $\bar{x}$ such that for all $x\in U_{\bar{x}}$, we have
	\begin{equation}
	\label{defn:grad-dominant-condition}
	f(x) - f(\bar{x})\leq \tau_f\|\grad f(x)\|^p,
	\end{equation}
	where $\tau_f>0$ is some universal constant independent of $\bar{x}$, then we call $f$ to be a locally gradient-dominant function of degree $p$.
\end{definition}
This definition stipulates that in a neighbourhood of a local minimum point, the function value is dominated by the size of Riemannian gradient.
As an example, consider the principal component analysis (PCA), which can be posed as
$$\min_X \,\,\, \langle A,XX^\top\rangle,~~ \mbox{subject to }  X\in \St_{n,r}.$$
The objective function of the problem actually satisfies the gradiant domination property with degree $p=2$. One can further prove that every second-order $\epsilon$-stationary point is close to the global optimum and every exact second-order stationary point is a global minimum.

\begin{theorem}
	\label{theorem:Grad-Domnt-speedups}
	Let the sequence $\{x_k,\xi_k\}$ be generated by Algorithm \ref{algo:Cubic}. Assume that $\{x_k\}$ is converging to a local minimum $\bar{x}$ and the whole sequence $\{x_k\}$ lies within the neighbourhood $U_{\bar{x}}$ where the objective function $f$ is locally gradient-dominant with degree $p$. Define
	\begin{equation}
	\label{theorem:Grad-Domnt-speedups-zk}
	z_k = \tau_f^{\frac{3}{2p-3}}\left(4/\tau_1\right)^{\frac{2p}{2p-3}} \tau_2^{\frac{3p}{2p-3}}\cdot(f(x_k)-f(\bar{x})),
	\end{equation}
	where $\tau_1, \tau_2$ are defined in Theorem \ref{theorem:Manifold-complexity}. We have: %It holds that % 	the convergence result follows that \vspace{-0.5cm}
	\begin{itemize}
		\item For $p = \frac{3}{2}$, it holds that %$\{z_k\}$ converges linearly to 0:
		$z_k\leq (\half)^kz_{0}$.
		\item For $1\leq p<\frac{3}{2}$, it holds that:
		%the sequence $\{z_k\}$ first converges super-linearly to the interval $(0,1)$, and then it converges sub-linearly to 0. \vspace{-0.3cm}
		\begin{itemize}
			\item if $z_0\geq 2^\frac{3}{3-2p}$, then $z_k \leq z_{k-1}^\frac{2p}{3}$;
			%Namely, a super-linear convergence $\ln z_k\leq (\frac{2p}{3})^k\ln z_0$ is obtained.
			\item if $z_{0}\leq 2^\frac{3}{3-2p}$, then after at most $t = 1+\lceil\frac{3}{3-2p}\rceil$ steps, we have $z_{t}<1$;
			\item if $z_0 < 1$, then by letting $\beta = \frac{2p}{3-2p}$ we have $z_k\leq \frac{1}{((1-2^{-1/\beta})k+1)^\beta} = \mathcal{O}(k^{-\beta})$. In particular, when $p=1$ then $z_k\leq\mathcal{O}(k^{-2})$. %\vspace{-0.3cm}
		\end{itemize}
		\item For $p>\frac{3}{2}$, it holds that:
		%the sequence $\{z_k\}$ first converges linearly to the interval $(0,1)$, and then it converges super-linearly to 0. \vspace{-0.3cm}
		\begin{itemize}
			\item if $z_0\geq 1$, then $\frac{z_k}{z_{k-1}} \leq  \left(1+z_0^\frac{3-2p}{2p}\right)^{-1}<1$;  %\vspace{-0.2cm}
			\item if $z_0 < 1$, then $z_k \leq z_{k-1}^\frac{2p}{3}$. %Particularly, $\frac{3}{2} < p \le 2$ gives super-linear convergence of order $\frac{2p}{3}$.
		\end{itemize}
	\end{itemize}
\end{theorem} %\vspace{-0.6cm}
%\begin{remark}
We remark that the rate in Theorem \ref{theorem:Manifold-complexity} yields  $\|\grad f(x_k)\|\leq \mathcal{O}(k^{-\frac{2}{3}})$. Combined with the gradient-dominant condition of degree $p$, this results in a sub-linear convergence rate of $\mathcal{O}(k^{-\frac{2p}{3}})$, which is slower than the $\mathcal{O}(k^{-\frac{2p}{3-2p}})$ rate here with $1\leq p<\frac{3}{2}$.
%\end{remark}\vspace{-0.3cm}
\begin{proof}
	Recall \eqref{thm:Gen-convgce-4} and \eqref{thm:Gen-convgce-6} state that
	$$f(x_k)-f(x_{k+1})\geq \frac{\tau_1}{4} \|\xi_{k}\|^3\mbox{ and }\|\grad f(x_{k+1})\|\leq \tau_2\|\xi_{k}\|^2.$$
	These relationships indicate
	$$f(x_k)-f(x_{k+1})\geq \frac{\tau_1}{4}\tau_2^{-\frac{3}{2}}\|\grad f(x_{k+1})\|^{\frac{3}{2}}.$$
	Together with the gradient-dominant condition \eqref{defn:grad-dominant-condition}, we get
	\begin{equation*}
	f(x_k) - f(x_{k+1}) \geq  \frac{\tau_1}{4\tau_2^{\frac{3}{2}}}(\tau_f)^{-\frac{3}{2p}}(f(x_{k+1})-f(\bar{x}))^{\frac{3}{2p}}.
	\end{equation*}
	If we define $z_k$ according to the equation \eqref{theorem:Grad-Domnt-speedups-zk},  then the above inequality can be simplified to
	\begin{equation}
	\label{thm:Grad-Domnt-speedups-1}
	z_k \geq z_{k+1} + z_{k+1}^\frac{3}{2p}.
	\end{equation}
	Next let us discuss various values of $p$. %the results in the theorem in order.
	
	First, when $p = \frac{3}{2}$, \eqref{thm:Grad-Domnt-speedups-1} becomes
	$z_k\geq 2z_{k+1}$
	and the result follows.
	
	Second, consider $1\leq p<\frac{3}{2}$. Suppose $z_k\geq 1$. If $z_{k+1}<1$, this stage is over; otherwise, $z_{k+1}\geq 1$, and since $\frac{3}{2p}>1$ we have $z_{k+1}^{^\frac{3}{2p}}\geq z_{k+1}$. Then \eqref{thm:Grad-Domnt-speedups-1} implies that
	$z_k\geq \max\{2z_{k+1}, z_{k+1}^\frac{3}{2p}\}.$
	In case $z_{k+1}\geq 2^{\frac{3}{3-2p}}$, we have $2z_{k+1}\leq z_{k+1}^\frac{3}{2p}$ and so the inequality $z_k\geq  z_{k+1}^\frac{3}{2p}$ dominates. Consequently, in that scenario after $t$ steps, we have $ \ln z_t \leq \left(\frac{2p}{3}\right)^t \ln z_0.$
	
	Whenever $z_{k_0}\leq 2^{\frac{3}{3-2p}}$  for some $k_0$, then $z_k\geq  2z_{k+1}$ starts to dominate henceforth.
	Therefore, $z_{k_0+t}$ is guaranteed to be less than 1 when $t \geq 1+\lceil\frac{3}{3-2p}\rceil$.
	
	Now, for simplicity suppose $z_0<1$. Letting $\beta = \frac{2p}{3-2p}$, \eqref{thm:Grad-Domnt-speedups-1} leads to
	\begin{equation}
	\label{thm:Grad-Domnt-speedups-2}
	\left(\frac{1}{z_k}\right)^\frac{1}{\beta}\leq \left(\frac{1}{z_{k+1}}\right)^{\frac{1}{\beta}}\left( 1+z_{k+1}^{\frac{1}{\beta}}\right)^{-\frac{1}{\beta}}.
	\end{equation}
	Note that function $r(s):=(1+s)^{-1/\beta}$ is strictly convex in $s$. Hence for any $0<s<1$, we have
	$r(s)<r(0)+\frac{r(1)-r(0)}{1-0}s = 1-(1-2^{-1/\beta})s.$
	Substituting this inequality into \eqref{thm:Grad-Domnt-speedups-2} with $s = z_{k+1}^{1/\beta}$ we obtain
	$$\left(\frac{1}{z_k}\right)^\frac{1}{\beta}\leq \left(\frac{1}{z_{k+1}}\right)^{\frac{1}{\beta}} - (1-2^{-\frac{1}{\beta}}),$$
	which further implies that
	\begin{equation}
	\label{thm:Grad-Domnt-speedups-3}
	z_k\leq \frac{1}{((1-2^{-1/\beta})k+z_0^{-1/\beta})^\beta} \leq \frac{1}{((1-2^{-1/\beta})k+1)^\beta} = \mathcal{O}(k^{-\beta}).
	\end{equation}
	This completes our analysis for the case $p<\frac{3}{2}$.
	
	For the case $p>\frac{3}{2},$ when $z_0 \geq 1$, \eqref{thm:Grad-Domnt-speedups-1} immediately leads to
	$$\frac{z_{k+1}}{z_k}\leq (1+z_{k+1}^{\frac{3-2p}{2p}})^{-1}\leq \left(1+z_{0}^{\frac{3-2p}{2p}}\right)^{-1}<1,$$
	which is a linear rate of convergence to the interval $(0,1)$. Whenever $z_0<1$, then \eqref{thm:Grad-Domnt-speedups-1} yields
	$$z_{k+1}\leq z_k^{\frac{2p}{3}}$$
	where $\frac{2p}{3}>1$. The theorem is proven.
\end{proof}

\subsubsection{Nondegenerate Riemannian Hessian} $\mbox{ }$ A$\mbox{ }$ second condition under  which a faster local convergence holds is when a local minimum point has positive definite Riemannian Hessian. Under this condition, it maintains a local quadratic rate of convergence, which is typical for the Newton type methods in the usual Euclidean case. It is interesting to note that this property carries over to Riemannian optimization as well.

Formally, let us call a second-order stationary point $\bar{x}$ to be {\it non-degenerate}\/ if there exists a constant $\delta>0$ such that for any $\xi\in\cT_{\bar{x}}\cM$, we have
\begin{equation}
\label{defn:non-degenerate-Hess}
\langle \Hess f(\bar{x})[\xi],\xi\rangle \geq \delta\|\xi\|^2.
\end{equation}
%\vspace{-1cm}
\begin{theorem}[Local Quadratic Convergence]
	Let the sequence $\{x_k,\xi_{k}\}$ be generated by Algorithm \ref{algo:Cubic}. Suppose $x_k\rightarrow\bar{x}$ where $\bar{x}$ is a nondegenerate local minimum satisfying \eqref{defn:non-degenerate-Hess}. %Let $U_{\bar{x}}\subset\cM$ be a neighbourhood of $\bar{x}$ where for $\forall x\in U_{\bar{x}}$ and for  $\forall\xi\in\cT_x\cM$, we have
	Suppose that $x_0$ satisfies \eqref{defn:non-degenerate-Hess} with some constant $\delta_0>0$  and $\|\xi_0\|\leq\min\left\{\frac{3d_{\cM}}{5L_1},\frac{3\delta_0}{10DL_1},\frac{\delta_0/4}{\tau_2}\right\}$, where $D, d_{\cM}$ are defined in Lemma \ref{lemma:Lip-Rie-Hess}, and $L_1$ is defined in Proposition \ref{prop:retraction-regularity}.
	% and $x_0\in U_{\bar{x}}$
	Then, a quadratic rate of convergence holds: $$\frac{\tau_2}{\delta_0/2}\|\xi_{k+1}\| \leq \left(\frac{\tau_2}{\delta_0/2}\|\xi_{k}\|\right)^2,$$
	%and consequently $\|\xi_{k}\|\leq\frac{\sigma+L_H}{\delta_0/2}\cdot(\frac{\delta_0/2}{L_H+\sigma}\|\xi_{0}\|)^{2^k}$,
	where $\tau_2$ is defined in Theorem \ref{theorem:Manifold-complexity}. As a result,
	$$
	\|\grad f(x_{k})\|\leq
	%(L_H+\sigma)\|\xi_{k-1}\|^2 =
	\mathcal{O}(2^{-2^k}), \mbox{ and }
	\langle \Hess f(x_k)[\xi],\xi\rangle \geq \frac{\delta_0}{2}\|\xi\|^2, \,\, \forall \xi\in\cT_{x_k}\cM.
	$$
\end{theorem}
\begin{proof}
	%	First we justify the validity of the assumption on $\|\xi_{0}\|$.
	Since $\bar{x}$ satisfies the non-degeneracy condition \eqref{defn:non-degenerate-Hess} and $\Hess f$ is Lipschitz continuous in a neighbourhood of $\bar{x}$, there exists a neighbourhood $U_{\bar{x}}\subset\cM$ of $\bar{x}$ and a positive constant $0<\delta\leq \delta_0$, such that for any $x\in U_{\bar{x}}$, \eqref{defn:non-degenerate-Hess} is satisfied for this $\delta$. Since $\{x_k\}$ converges to $\bar{x}$ and $\|\xi_{k}\|$ converges to 0, the condition regarding $\xi_{0}$ will be satisfied for some $\xi_{k}$. %then picking this $x_k$ as the new $x_0$ justifies our assumption.
	One may redefine this $x_k$ to be $x_0$, and the condition on $\xi_0$ is then satisfied.
	
	Let us proceed to the proof of the theorem.
	%	for the ease of notation.	
	Now the assumption implies that
	$\frac{2\tau_2}{\delta_0}\|\xi_{0}\|\leq\half<1$ and $\langle \Hess f(x_0)[\xi],\xi\rangle \geq \delta_0\|\xi\|^2\geq \half\delta_0\|\xi\|^2,  \,\, \forall \xi\in\cT_{x_0}\cM$. We claim that (will prove this claim in one moment) for any nonnegative integer $k$, %For simplicity, suppose that for the whole sequence $\{x_k\}$ we have
	\begin{equation}
	\label{thm:Quadratic-1}
	\langle \Hess f(x_k)[\xi],\xi\rangle \geq  \half\delta_0\|\xi\|^2, \,\,\forall \xi\in\cT_{x_k}\cM.
	\end{equation}
	
	Notice that \eqref{thm:convgce-1} and \eqref{thm:Gen-convgce-6} give $(\Hess f(x_k)+\frac{\sigma}{2}\|\xi_{k}\|I)[\xi_{k}] = g_k$ and $\|g_k\|\leq \tau_2\|\xi_{k-1}\|^2$. Therefore,
	$$\|\xi_{k}\|\leq \frac{\|g_k\|}{\frac{\delta_0}{2}+\frac{\sigma}{2}\|\xi_{k}\|}\leq \frac{\tau_2\|\xi_{k-1}\|^2}{\frac{\delta_0}{2}+\frac{\sigma}{2}\|\xi_{k}\|}\leq \frac{\tau_2\|\xi_{k-1}\|^2}{ \delta_0/2 }.$$
	Defining $z_k := \frac{2\tau_2}{\delta_0}\|\xi_{k}\|$, the above inequality is equivalent to 	$z_{k}\leq z_{k-1}^2.$
	Since $z_0\leq\half$, the whole sequence converge quadratically to 0. Specifically, this leads to %, this indicates that
	$$\|\xi_{k}\| = \frac{\delta_0}{2\tau_2}z_k\leq \frac{\delta_0}{2\tau_2}z_0^{2^k}\leq \frac{\delta_0}{2\tau_2}\cdot \frac{1}{2^{2^k}}. $$ %(\half)^{2^k}.$$
	By \eqref{thm:Gen-convgce-6}, we have
	$\|\grad f(x_k)\|\leq \tau_2\|\xi_{k-1}\|^2 = \mathcal{O}(2^{-2^k})$, as required.
	%	With our previous assumption \eqref{thm:Quadratic-1} the theorem is proved.
	%The therorem follows under Assumption \eqref{thm:Quadratic-1}.
	Now it remains only to show \eqref{thm:Quadratic-1}.	
	We shall prove by induction. The base case holds trivially. Suppose we already have
	$$\langle \Hess f(x_k)[\xi],\xi\rangle \geq  \half\delta_0\|\xi\|^2, \,\, \forall \xi\in\cT_{x_k}\cM, \mbox{ for } k\leq k_0-1,$$
	which means that
	$\|\xi_{k}\|\leq \frac{\delta_0}{2\tau_2}z_0^{2^k}\leq \frac{\delta_0}{2\tau_2}(z_0)^{2k}, \mbox{ for } k\leq k_0-1.$
	Therefore,
	\begin{small}
		\begin{eqnarray*}
			\|x_{k_0}-x_0\| & \leq & \sum_{k=1}^{k_0} \|x_k-x_{k-1}\|\leq  L_1\sum_{k=0}^{k_0-1} \|\xi_{k}\| \leq  \frac{L_1\delta_0}{2\tau_2}\left(z_0+\sum_{k=1}^{k_0-1} (z_0)^{2k}\right)  \leq  \frac{L_1\delta_0}{2\tau_2}\left(z_0+\frac{z_0^2}{1-z_0^2}\right).
		\end{eqnarray*}
	\end{small}Since $z_0\leq \half$, we have $\frac{z_0^2}{1-z_0^2}\leq \frac{\half z_0}{1-(\half)^2} = \frac{2}{3}z_0$, consequently
	$\|x_{k_0}-x_0\|\leq L_1\frac{\delta_0}{2\tau_2}\cdot\frac{5}{3}z_0.$
	Suppose $\|\xi_0\|$ satisfies the condition of the theorem, $\|x_{k_0}-x_0\|\leq d_{\cM}$. Then by Lemma \ref{lemma:Lip-Rie-Hess},
	$$\langle \Hess f(x_{k_0})[\xi],\xi\rangle \geq (\delta_0-D\|x_{k_0}-x_0\|)\|\xi\|^2\geq \half\delta_0\|\xi\|^2, \,\, \forall \xi\in\cT_{x_{k_0}}\cM,$$
	completing the proof.
\end{proof}

%\section{Example: Optimization on Stiefel Manifolds}\label{sec:app}
\section{The Case of Stiefel Manifolds} \label{sec:app}

%\textbf{Basics of Stiefel Manifold.}
When equipped with the standard Euclidean inner product, the so-called Stiefel manifold $\St_{n,r}$ is an $nr-\frac{r(r+1)}{2}$ dimensional Riemannian submanifold. Its tangent space is given by
%\begin{equation}
%\label{Stief:Tan-space1}
$\cT_X\St_{n,r} = \{Z\in\R^{n\times r}:\: X^\top Z+Z^\top X = 0\}.$
%\end{equation}
%Suppose we have found a matrix $X_\perp\in\R^{n\times n-r}$ such that $[X,X_\perp]$ is an $n\times n$ orthogonal matrix.
%Then an alternative characterization of $\cT_XSt_{n,r}$ is
%\begin{equation}
%\label{Stief:Tan-space2}
%\cT_XSt_{n,r} = \{Z\in\R^{n\times r}|Z = XA+X_\perp B, \forall A\in Skew(r,r), \forall B\in\R^{n-r\times r}\}.
%\end{equation}
%Using this characterization, one can determine the orthogonal projection $P_X$ onto $\tanst$ as
The orthogonal projection onto the tangent space is given by
\begin{eqnarray}
\label{Stief:Proj}
P_X(G)
%& = & X\cdot\frac{X^\top G-G^\top X}{2} +  X_\perp\cdot  X_\perp^\top G \nonumber\\
& = & G-\frac{1}{2}(XX^\top G+XG^\top X),~ \,\, \forall G\in\R^{n\times r}.
\end{eqnarray}
Consequently, the Riemannian gradient of $f$ at point $X\in \St_{n,r}$ equals $\grad f(X) = P_X(\nabla f(X)).$
%For $\forall Z\in\tanst$,  one can derive that
%\begin{eqnarray}
%\label{Stief:Rie-grad-diff}
%\mathrm{D}\grad f(X)[Z] & = & - \frac{ZX^\top\nabla f(X)+XZ^\top\nabla f(X)+XX^\top\nabla^2f(X)[Z]}{2}+\nabla^2f(X)[Z]\nonumber\\
%& & -\frac{Z\nabla f(X)^\top X + X(\nabla^2f(X)[Z])^\top X + X\nabla f(X)^\top Z}{2}.
%\end{eqnarray}
%Therefore, the operation of the  Riemannian Hessian on $Z$ is then
%$\Hess f(X)[Z] = P_X(\mathrm{D}\grad f(X)[Z]).$ In this paper,  we
Consider  the polar retraction defined as
\begin{equation}
\label{Stief:Retr-Polar}
\Retr(X,Z) = (X+Z)(I_r+Z^\top Z)^{-\frac{1}{2}}, \,\, \forall Z\in\cT_X\St_{n,r}.
\end{equation}
For $Z$ outside of $\tanst$, the extended polar retraction is defined by replacing $Z$ with $P_X(Z)$ in the above formula. It is worth noting that the polar retraction is a second-order retraction according to \cite{Retraction}. Moreover, for polar retraction, Proposition \ref{prop:retraction-regularity} is satisfied with  $L_1=1$ and $L_2 = \half$ according to \cite{Stief:SVRG-Jiang}.  For the algorithmic setup, the constants $L^R_H$ and $C_g$  are characterized in the following lemmas, whose proofs are lengthy and technical; they are in Appendices \ref{Appendix-C} and \ref{Appendix-D}. %and are hence delegated to the Appendix.

\begin{proposition}[Global Pullback Hessian Lipschitz Continuity]
	\label{lemma:Stief-pullback-Lips}
	%	Let the constants $G,\ell_H, \ell_f$ be defined the same as before.
	If $\Retr(\cdot,\cdot)$ is chosen to be the polar retraction on the Stiefel manifold $\St_{n,r}$ (or a products of Stiefel manifolds), then the condition \eqref{pullback-Lips} holds with
	\begin{equation}
	\label{lemma:Stief-pullback-Lips-1}
	L_H = 13.66G  + 12.55\ell_f + 4\ell_H,
	\end{equation}
	which is independent of $x\in\cM$ as well as the dimensions $n,r$.
\end{proposition}

%\begin{remark}
Though we only considered the local Lipschitz continuity of the pullback Hessian, we actually have a stronger global Lipschitz continuity property on the Stiefel manifold.
%With polar retraction being second-order retraction, the gradient steps in Algorithm \ref{algo:Cubic} is no longer needed.
%\end{remark}

\begin{proposition}
	\label{lemma:Stief-pullback-grad-Neibourhood}
	For the polar retraction on the Stiefel manifold $\St_{n,r}$ (or a product of Stiefel manifolds), Corollary \ref{corollary:pullback-grad-Neighbourhood} holds with constant
	$ C_g =\frac{1}{8.62}$.
	%If the $Retr(\cdot,\cdot)$ is chosen to the QR retraction. Then, for $\forall x\in St_{n,r}$, for $\forall\xi\in\cT_xSt_{n,r}$, we only need to choose $\|\xi\|_F\leq\frac{1}{6+\sqrt{2}+\sqrt{10}}$ to make \eqref{theorem:pullback-grad-Neighbourhood-1} hold again.
\end{proposition}

\subsection{Numerical experiments}
As a numerical illustration, we consider the following problem
(cf.~\cite{APP:SDP_Gro_IEQ}):
\begin{eqnarray}
\label{prob:MOC-SDP}
\max  \,\,\, \langle A, UU^\top \rangle \,\, \mbox{ subject to }   u_i^\top u_i = I_{d\times d},~~ i = 1,...,n,
\end{eqnarray}
where $A\in\R^{dn\times dn}$, $u_i\in\R^{k\times d}$, $U = \left[u_1,\cdots,u_n\right]^\top\in\R^{dn\times k}.$ When $d = 1$, $u_i^\top u_i = I_{d\times d} = 1,$ the problem is constrained on a product of $n$ spheres, and the problem is called the \emph{low-rank max-cut-SDP} problem. When $d\geq 2$, the problem is constrained on a product of $n$ Stiefel maninfolds, and the problem is called the \emph{low-rank max-orthogonal-cut-SDP} problem. In the experiments, we test our algorithm in the cases where $d = 1$ and  $d = 3$ respectively.

We sample the matrices $A\sim \mathrm{GOE}(300),$ where $A\sim\mathrm{GOE}(n)$ stands for a matrix with $A_{ij}\sim N(0,1/n)$ when $i\neq j$ and $A_{ii}\sim N(0,2/n)$ (see \cite{APP:SDP_Gro_IEQ}).
In the implementation of Algorithm \ref{algo:Cubic}, we follow \cite{Cubic:Carmon} in that the subproblems are solved approximately with a fixed stepsize gradient method. However, we change the stopping criterion of the subproblem to be $\|\nabla\hat{m}_{x_k,\sigma}(\xi_t)\|\leq c\|\mathrm{grad}f(x_k)\|$, where $c$ is some constant.
%In general, the algorithm in \cite{Cubic:Carmon} takes $\mathcal{O}(1/\epsilon^2)$ gradient and Hessian-vector product evaluations to reach a second-order $\epsilon$-stationary point, which is the same as those required by the gradient descent method.
The performance of the algorithm is plotted in Figure 1 and Figure 2, where Algorithm \ref{algo:Cubic} (CRRN), the Riemannian trust-region (RTR) algorithm \cite{RM_glo} and the Riemannian gradient descent (RGD) algorithm are compared. For each of $d=1$ and $d=3$ problems, we randomly generate 3 initial solutions and show the performance of the three algorithms on these cases.  
%\begin{figure} [h]
%	\label{fig:Num-1}
%	\caption{Experiments on low-rank Max-Cut-SDP ($d=1$) problem}
%	\begin{minipage}[t]{0.45\textwidth}
%		\includegraphics[width=0.9\linewidth]{d=1-fig1.eps}
%	\end{minipage}
%	\hspace{\fill}
%	\begin{minipage}[t]{0.45\textwidth}
%		\includegraphics[width=0.9\linewidth]{d=1-fig2.eps}
%	\end{minipage}

%	\vspace*{0.1cm} % (or whatever vertical separation you prefer)
%	\begin{minipage}[t]{0.45\textwidth}
%		\includegraphics[width=0.9\linewidth]{d=1-fig3.eps}
%	\end{minipage}
%	\hspace{\fill}
%	\begin{minipage}[t]{0.45\textwidth}
%		\includegraphics[width=0.9\linewidth]{d=1-fig4.eps}
%	\end{minipage}
%\end{figure}

\begin{figure} 
	\label{fig:Num-1}
	\caption{Experiments on low-rank Max-Cut-SDP ($d=1$) problem}
	\centering
	\begin{tabular}{rrr}
		\begin{minipage}[t]{0.31\textwidth}
			\includegraphics[width=1\linewidth]{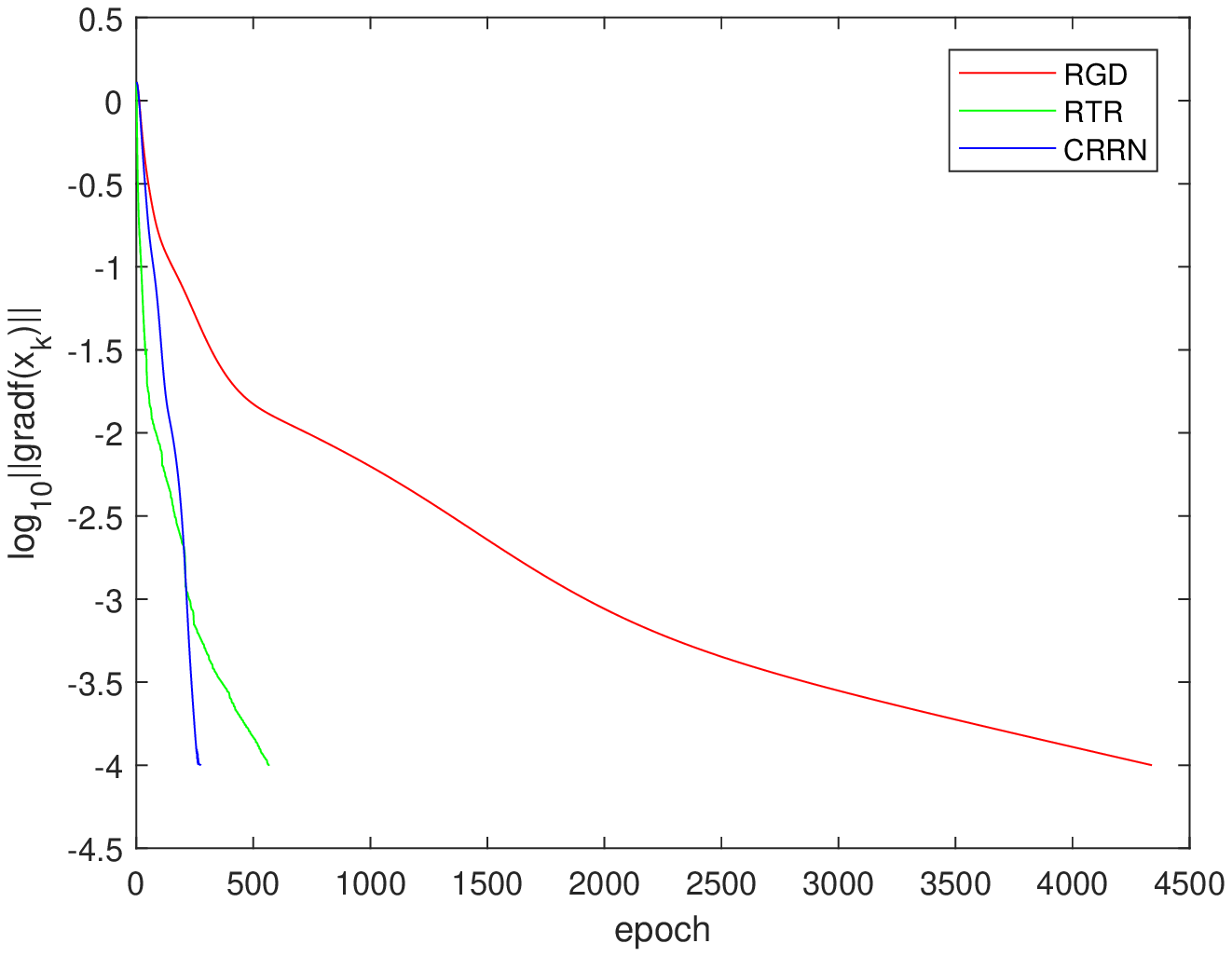}
		\end{minipage}
		%	\hspace{\fill}
		\begin{minipage}[t]{0.31\textwidth}
			\includegraphics[width=1\linewidth]{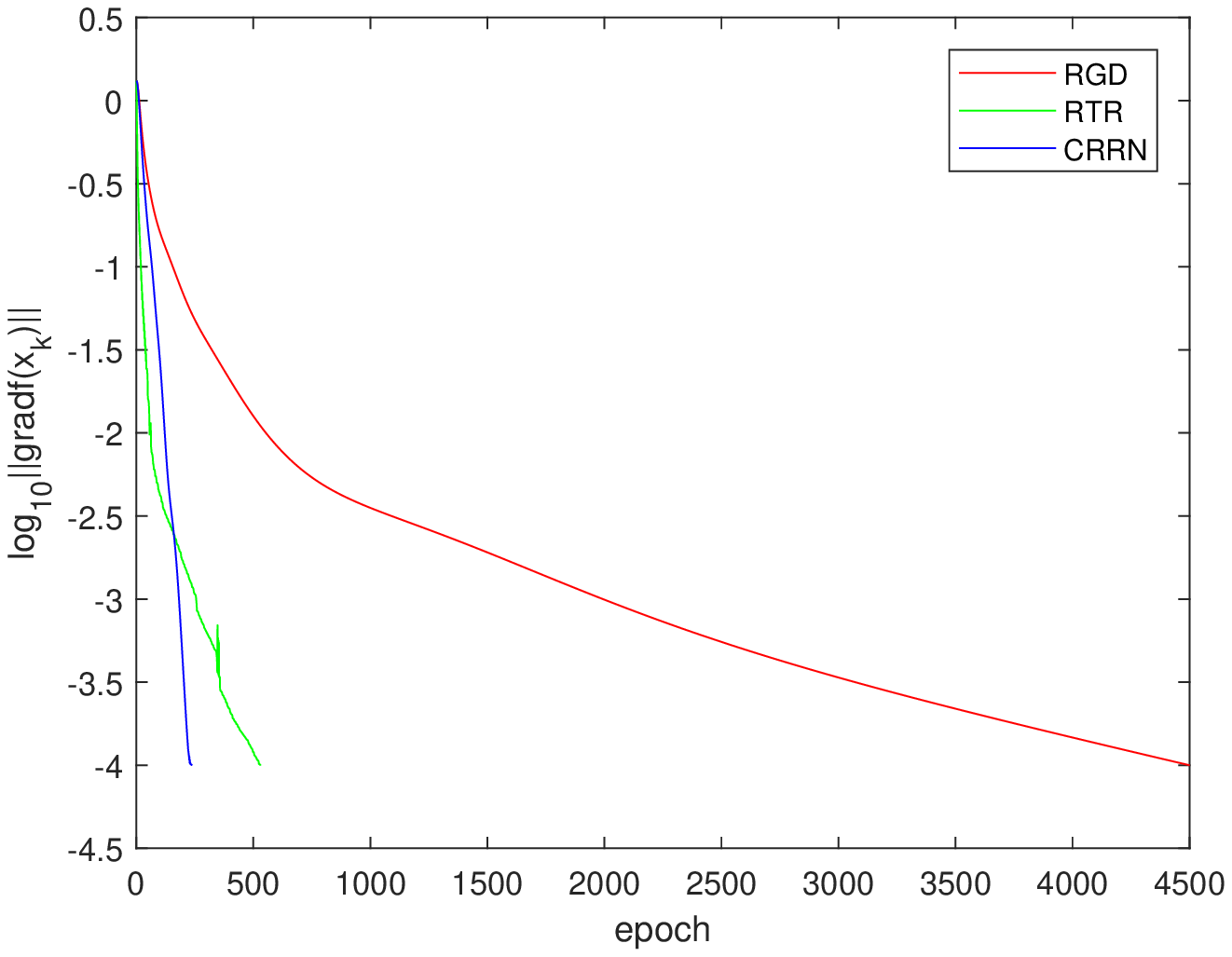}
		\end{minipage}
		
		%	\vspace*{0.1cm} % (or whatever vertical separation you prefer)
		\begin{minipage}[t]{0.31\textwidth}
			\includegraphics[width=1\linewidth]{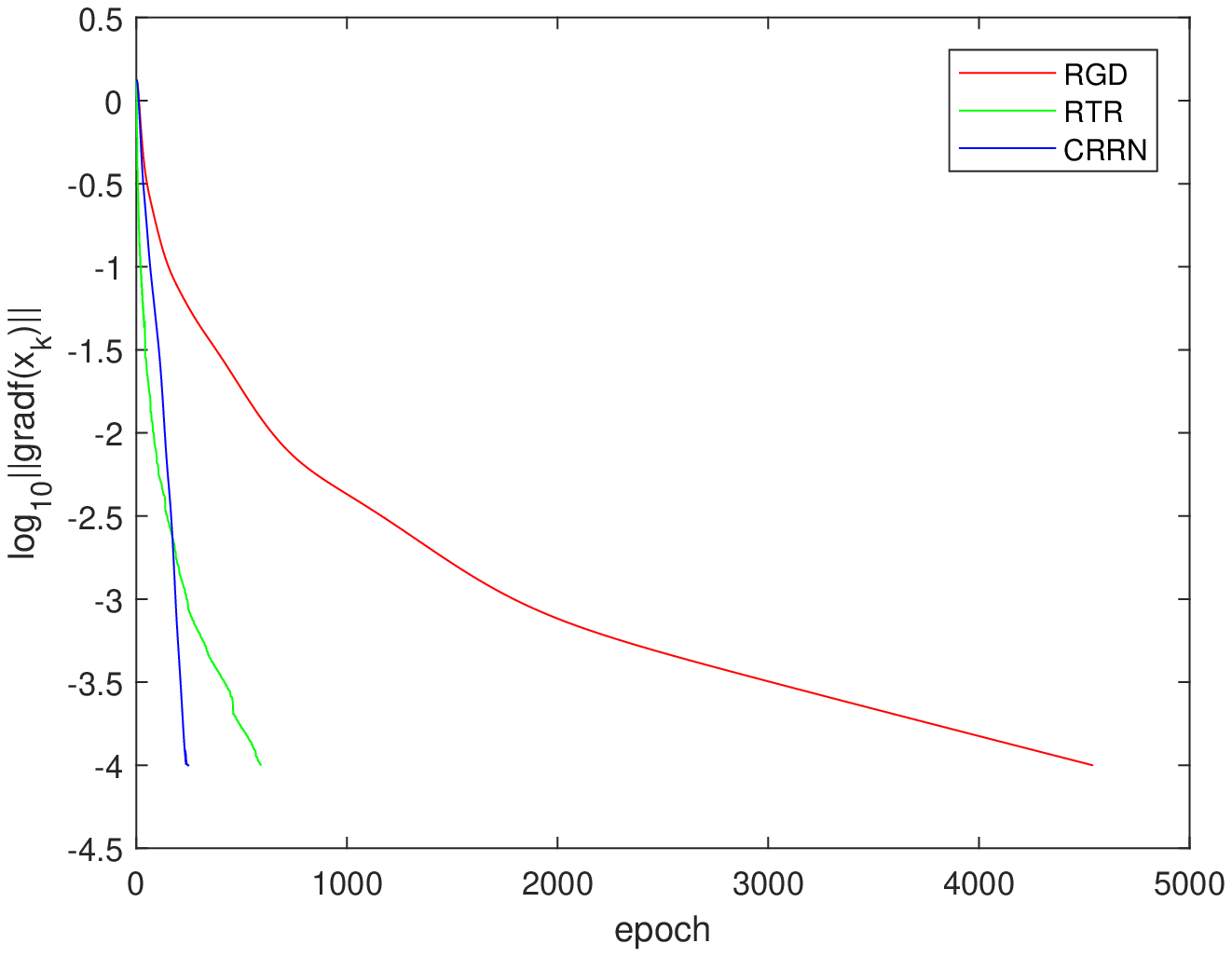}
		\end{minipage}
		%	\hspace{\fill}
		%\begin{minipage}[t]{0.25\textwidth}
		%	\includegraphics[width=1\linewidth]{d=1-fig4.eps}
		%\end{minipage}   
	\end{tabular}
\end{figure}

\begin{figure} 
	\label{fig:Num-2}
	\caption{Experiments on low-rank Max-Orthogonal-Cut-SDP ($d=3$) problem}
	\centering
	\begin{tabular}{rrr}
		\begin{minipage}[t]{0.31\textwidth}
			\includegraphics[width=1\linewidth]{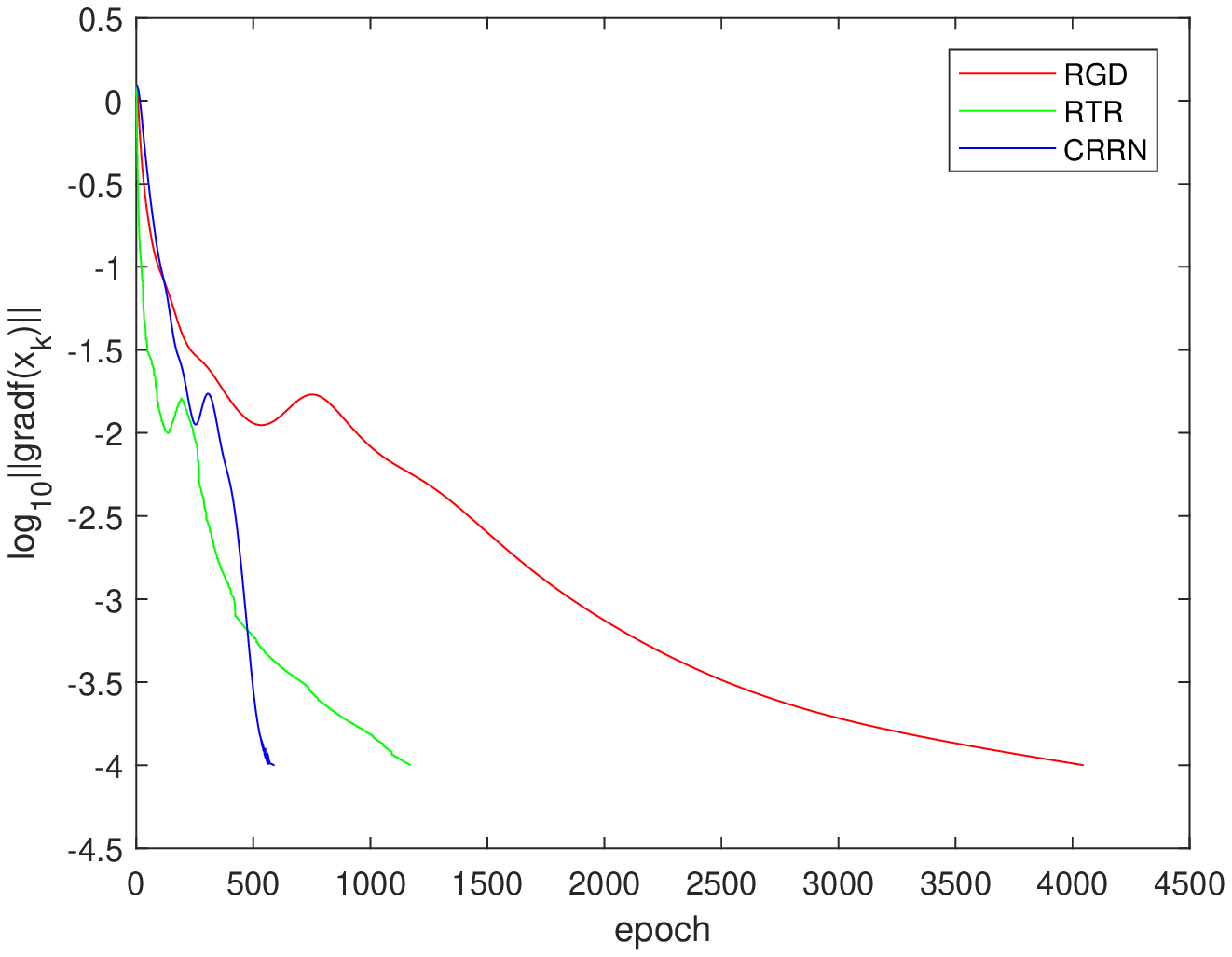}
		\end{minipage}
		%\hspace{\fill}
		\begin{minipage}[t]{0.31\textwidth}
			\includegraphics[width=1\linewidth]{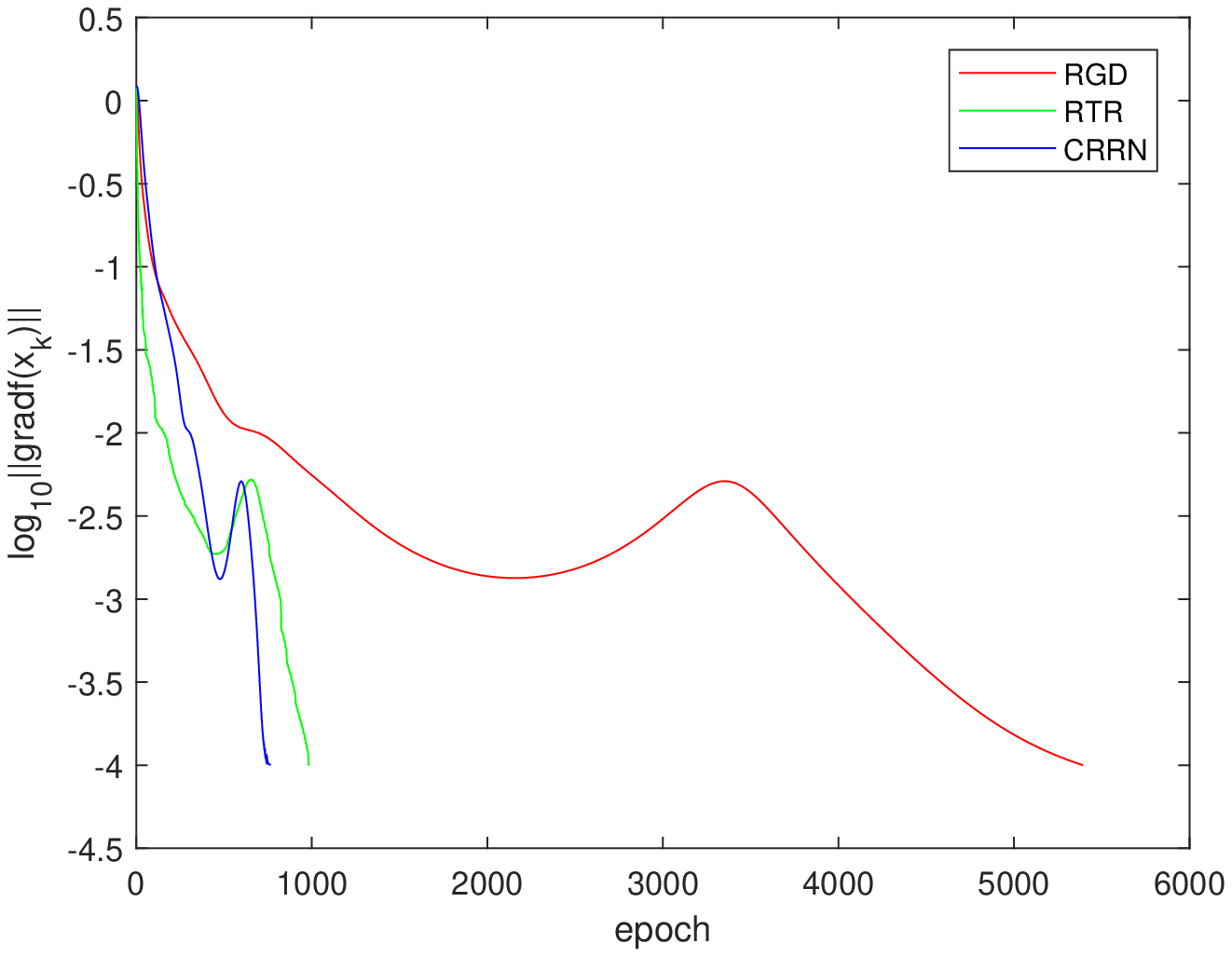}
		\end{minipage}
		
		%\vspace*{0.1cm} % (or whatever vertical separation you prefer)
		\begin{minipage}[t]{0.31\textwidth}
			\includegraphics[width=1\linewidth]{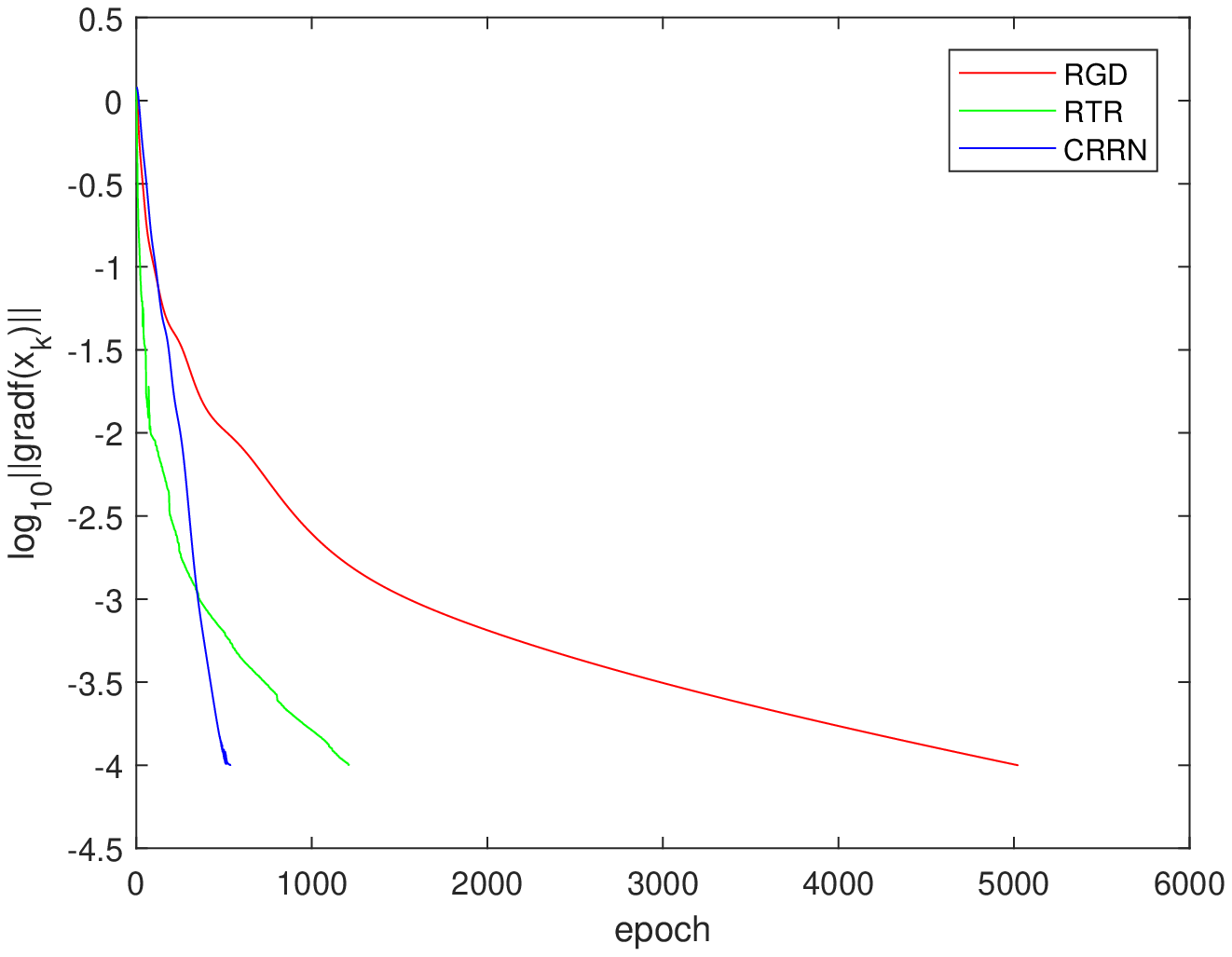}
		\end{minipage}
		%\hspace{\fill}
		%\begin{minipage}[t]{0.25\textwidth}
		%	\includegraphics[width=0.9\linewidth]{d=3-fig4.eps}
		%\end{minipage}  
	\end{tabular}
\end{figure}

%\begin{figure} [h]
%	\label{fig:Num-2}
%	\caption{Experiments on low-rank Max-Orthogonal-Cut-SDP ($d=3$) problem}
%	\begin{minipage}[t]{0.45\textwidth}
%		\includegraphics[width=0.9\linewidth]{d=3-fig1.eps}
%	\end{minipage}
%	\hspace{\fill}
%	\begin{minipage}[t]{0.45\textwidth}
%		\includegraphics[width=0.9\linewidth]{d=3-fig2.eps}
%	\end{minipage}
%	
%	\vspace*{0.1cm} % (or whatever vertical separation you prefer)
%	\begin{minipage}[t]{0.45\textwidth}
%		\includegraphics[width=0.9\linewidth]{d=3-fig3.eps}
%	\end{minipage}
%	\hspace{\fill}
%	\begin{minipage}[t]{0.45\textwidth}
%		\includegraphics[width=0.9\linewidth]{d=3-fig4.eps}
%	\end{minipage}
%\end{figure}

Finally, we show two examples where CRRN converges superlinearly to a local minimum point for the case $d=1$ in Figure 3. The data is from the Caltech students Facebook social network datasest. The network consists of 597 nodes. In these two cases, we set $k = 6$ and $k = 8$ respectively. 
\begin{figure} 
	\label{fig:Num-3}
	\caption{Superlinear convergence example on low-rank Max-Cut-SDP ($d=1$) problem}
	\centering
	\begin{tabular}{rr}
		\begin{minipage}[t]{0.45\textwidth}
			\includegraphics[width=1\linewidth]{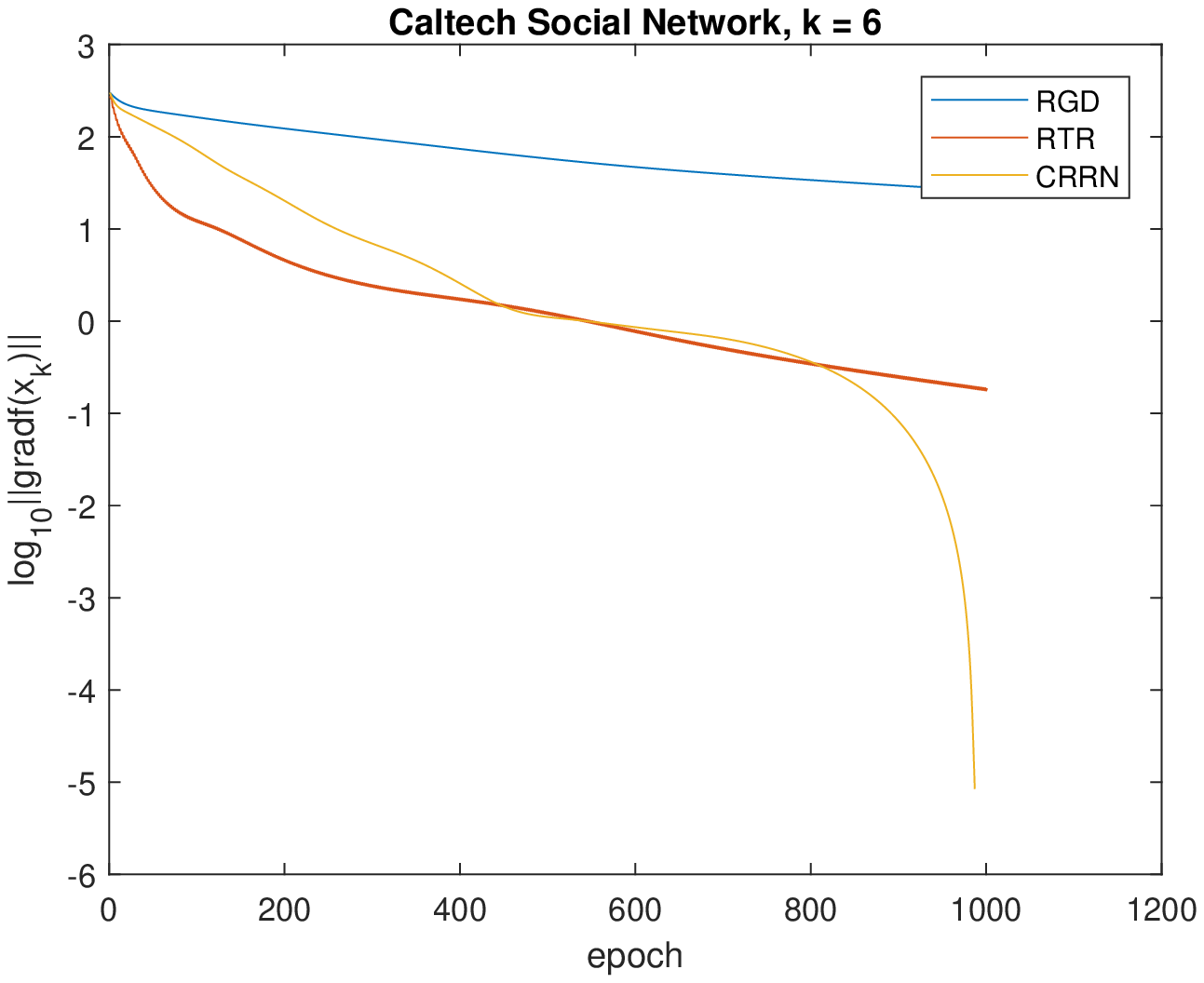}
		\end{minipage}
		%	\hspace{\fill}
		\begin{minipage}[t]{0.45\textwidth}
			\includegraphics[width=1\linewidth]{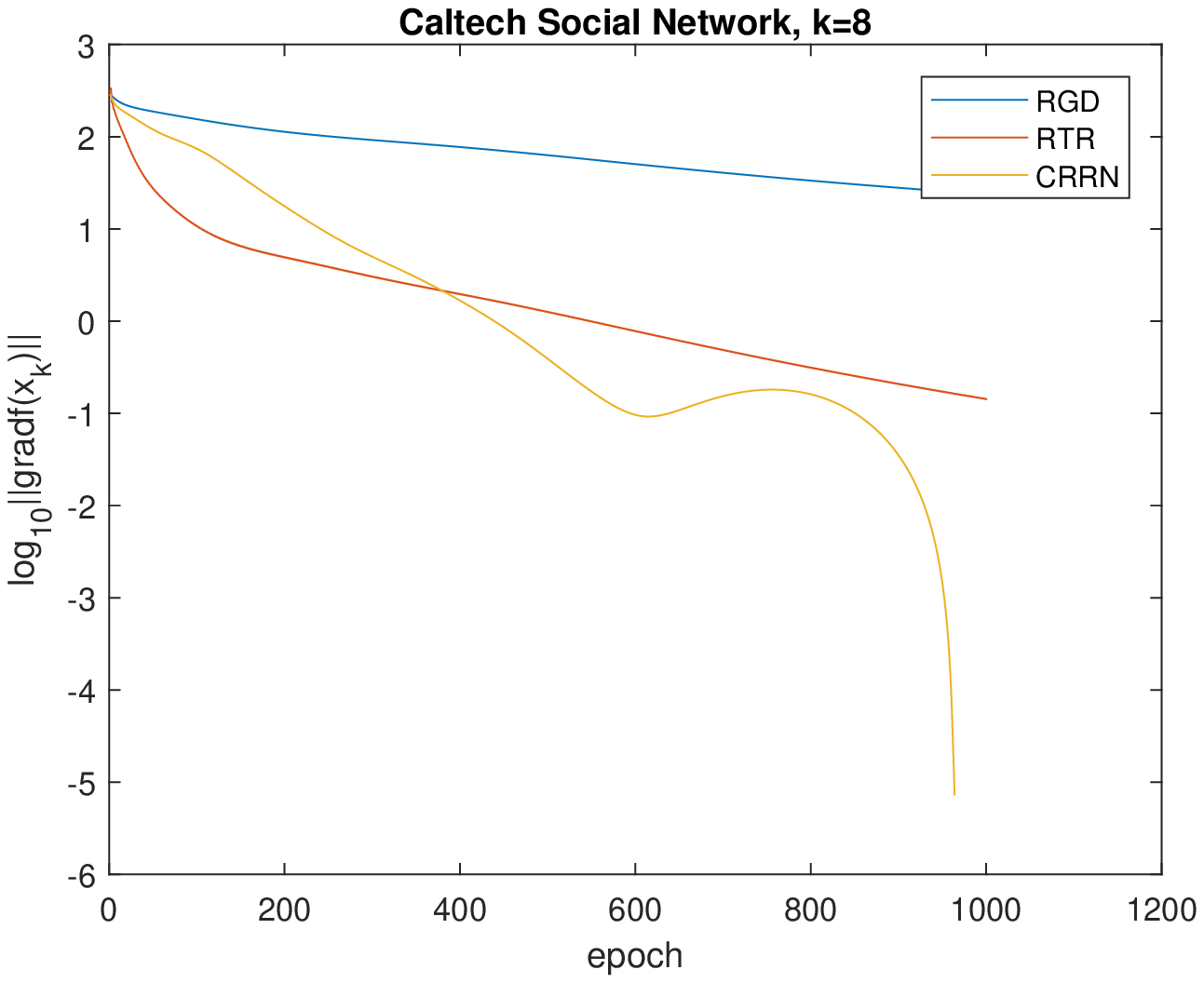}
		\end{minipage}
		
		%	\vspace*{0.1cm} % (or whatever vertical separation you prefer)
		%	\begin{minipage}[t]{0.31\textwidth}
		%	\includegraphics[width=1\linewidth]{d=1-fig3.eps}
		%\end{minipage} 
	\end{tabular}
\end{figure}

As is observed in \cite{APP:SDP_Gro_IEQ}, the gradient descent algorithm with fixed stepsize actually works very well for this problem and finally converges to a second-order stationary point. Similar observations can be made on the general behaviors of the second-order methods. In terms of Riemannian gradient size, the CRRN runs slight slower than the RGD and RTR at first, and then quickly catches up and takes over both algorithms as the Riemannian gradient gets smaller.
%This actually motivate us to discover more efficient algorithm design for the subproblems as well as the implementations.

\section{Conclusion}\label{sec:conclusion}
In this paper, we extend Nesterov's cubic regularized Newton's method to Riemannian optimization. It is shown that under mild conditions on the objective function and the Riemannian manifold, an $\mathcal{O}(1/\epsilon^{\frac{3}{2}})$ complexity bound can be guaranteed.
This establishes that Riemannian optimization essentially bears the same degree of iteration complexity as the classical unconstrained optimization over Euclidean space.
Specifically, we consider optimization over the Stiefel manifold (or a product of Stiefel manifolds). In this case, all the constants established in the general setting can be exactly computed. Our numerical experiments show that our method is competitive compared against the Riemannian gradient descent and the Riemannian trust-region method in terms of iteration complexity, although we notice that solving the subproblems using Carmon's gradient descent method is indeed more expensive than solving the Riemannian trust-region subproblem.
This motivates the study of effective schemes to solve the subproblems.
To compute the required constants/parameters maybe a non-trivial task. However, once it is done, then it will be valid for the entire class of manifolds. In this paper, we carried out this computations for the case of Stiefel manifolds.
It will be interesting to design an adaptive and parameter-free computational scheme, which is a topic for the future research.

\bibliography{Cubic_Stiefel}

\begin{thebibliography}{10}

\bibitem{RM_tru1}
{\sc P.~A. Absil, C.~G. Baker, and K.~A. Gallivan}, {\em Convergence analysis
  of {R}iemannian trust-region methods}, Technical Report,  (2006).

\bibitem{RM_tru}
\leavevmode\vrule height 2pt depth -1.6pt width 23pt, {\em Trust-region methods
  on {R}iemannian manifolds}, Foundations of Computational Mathematics, 7
  (2007), pp.~303--330.

\bibitem{Opt_Manif:Absil-etal-2009}
{\sc P.~A. Absil, R.~Mahony, and R.~Sepulchre}, {\em Optimization algorithms on
  matrix manifolds}, Princeton University Press, 2009.

\bibitem{Retraction}
{\sc P.~A. Absil and J.~Malick}, {\em Projection-like retractions on matrix
  manifolds}, SIAM Journal on Optimization, 22 (2012), pp.~135--158.

\bibitem{Cubic:Agarwal}
{\sc N.~Agarwal, Z.~Allen-Zhu, B.~Bullins, E.~Hazan, and T.~Ma}, {\em Finding
  approximate local minima faster than gradient descent}, in Proceedings of the
  49th Annual ACM SIGACT Symposium on Theory of Computing, ACM, 2017,
  pp.~1195--1199.

\bibitem{APP:Stie-products}
{\sc N.~Boumal}, {\em A {R}iemannian low-rank method for optimization over
  semidefinite matrices with block-diagonal constraints}, arXiv preprint
  arXiv:1506.00575,  (2015).

\bibitem{RM_glo}
{\sc N.~Boumal, P.~A. Absil, and C.~Cartis}, {\em Global rates of convergence
  for nonconvex optimization on manifolds}, arXiv preprint arXiv:1605.08101,
  (2016).

\bibitem{Gro:2}
{\sc N.~Boumal, V.~Voroninski, and A.~Bandeira}, {\em The non-convex
  {B}urer-{M}onteiro approach works on smooth semidefinite programs}, in
  Advances in Neural Information Processing Systems, 2016, pp.~2757--2765.

\bibitem{LRSDP-SamBurer}
{\sc S.~Burer and R.~D.~C. Monteiro}, {\em A nonlinear programming algorithm
  for solving semidefinite programs via low-rank factorization}, Mathematical
  Programming, 95 (2003), pp.~329--357.

\bibitem{Cubic:Carmon}
{\sc Y.~Carmon and J.~C. Duchi}, {\em Gradient descent efficiently finds the
  cubic-regularized non-convex {N}ewton step}, arXiv preprint arXiv:1612.00547,
   (2016).

\bibitem{Cubic:Cartis-1}
{\sc C.~Cartis, N.~I.~M. Gould, and P.~L. Toint}, {\em Adaptive cubic
  regularisation methods for unconstrained optimization. {P}art {I}:
  motivation, convergence and numerical results}, Mathematical Programming, 127
  (2011), pp.~245--295.

\bibitem{Cubic:Cartis-2}
\leavevmode\vrule height 2pt depth -1.6pt width 23pt, {\em Adaptive cubic
  regularisation methods for unconstrained optimization. {P}art {II}:
  worst-case function-and derivative-evaluation complexity}, Mathematical
  Programming, 130 (2011), pp.~295--319.

\bibitem{Cubic:Cartis-3}
\leavevmode\vrule height 2pt depth -1.6pt width 23pt, {\em An adaptive cubic
  regularization algorithm for nonconvex optimization with convex constraints
  and its function-evaluation complexity}, IMA Journal of Numerical Analysis,
  32 (2012), pp.~1662--1695.

\bibitem{GD-Exp-Cvg}
{\sc S.~S. Du, C.~Jin, J.~D. Lee, M.~I. Jordan, A.~Singh, and B.~Poczos}, {\em
  Gradient descent can take exponential time to escape saddle points}, in
  Advances in Neural Information Processing Systems, 2017, pp.~1067--1077.

\bibitem{GeoStif}
{\sc A.~Edelman, T.~A. Arias, and S.~Smith}, {\em The geometry of algorithms
  with orthogonality constraints}, SIAM Journal on Matrix Analysis and
  Applications, 20 (1998), pp.~303--353.

\bibitem{Book-Matrix-2012}
{\sc R.~A. Horn and C.~R. Johnson}, {\em Matrix analysis}, Cambridge University
  Press, 2012.

\bibitem{Stief:SVRG-Jiang}
{\sc B.~Jiang, S.~Ma, A.~M.-C. So, and S.~Zhang}, {\em Vector transport-free
  {SVRG} with general retraction for {R}iemannian optimization: Complexity
  analysis and practical implementation}, Preprint available at
  https://arxiv.org/abs/1705.09059,  (2017).

\bibitem{Stief:SVRG-Kasai}
{\sc H.~Kasai, H.~Sato, and B.~Mishra}, {\em {R}iemannian stochastic variance
  reduced gradient on {G}rassmann manifold}, arXiv preprint arXiv:1605.07367,
  (2016).

\bibitem{Dynamical-system-3}
{\sc J.~Lee, I.~Panageas, G.~Piliouras, M.~Simchowitz, M.~Jordan, and
  B.~Recht}, {\em First-order methods almost always avoid saddle points}, arXiv
  preprint arXiv:1710.07406,  (2017).

\bibitem{Dynamical-system-1}
{\sc J.~D. Lee, M.~Simchowitz, M.~I. Jordan, and B.~Recht}, {\em Gradient
  descent only converges to minimizers}, in Conference on Learning Theory,
  2016, pp.~1246--1257.

\bibitem{Smooth_Manif:Lee-John-2008}
{\sc J.~M. Lee}, {\em Introduction to smooth manifolds}, 41 (2008), p.~573.

\bibitem{APP:RM_sto4}
{\sc H.~Liu, W.~Wu, and A.~M.-C. So}, {\em Quadratic optimization with
  orthogonality constraints: Explicit {L}ojasiewicz exponent and linear
  convergence of line-search methods}, arXiv preprint arXiv:1510.01025,
  (2015).

\bibitem{APP:SDP_Gro_IEQ}
{\sc S.~Mei, T.~Misiakiewicz, A.~Montanari, and R.~I. Oliveira}, {\em Solving
  {SDP}s for synchronization and maxcut problems via the {G}rothendieck
  inequality}, arXiv preprint arXiv:1703.08729,  (2017).

\bibitem{Gro:1}
{\sc A.~Montanari}, {\em A {G}rothendieck-type inequality for local maxima},
  arXiv preprint arXiv:1603.04064,  (2016).

\bibitem{APP:nemirv}
{\sc A.~Nemirovski}, {\em Sums of random symmetric matrices and quadratic
  optimization under orthogonality constraints}, Mathematical Programming, 109
  (2007), pp.~283--317.

\bibitem{Cubic:Nestrov}
{\sc Yu. Nesterov and B.~T. Polyak}, {\em Cubic regularization of {N}ewton
  method and its global performance}, Mathematical Programming, 108 (2006),
  pp.~177--205.

\bibitem{Dynamical-system-2}
{\sc I.~Panageas and G.~Piliouras}, {\em Gradient descent converges to
  minimizers: The case of non-isolated critical points}, CoRR, abs/1605.00405,
  (2016).

\bibitem{Dynamical-system-4}
{\sc E.~Song, Z.~Shen, and Q.~Shi}, {\em Block coordinate descent only converge
  to minimizers}, arXiv preprint arXiv:1710.09047,  (2017).

\bibitem{wenDouble}
{\sc Z.~Wen and W.~Yin}, {\em A feasible method for optimization with
  orthogonality constraints}, Mathematical Programming, 142 (2013),
  pp.~397--434.

\bibitem{RieOpt:Yang-etal-2012}
{\sc W.~H. Yang, L.-H. Zhang, and R.~Song}, {\em Optimality conditions for the
  nonlinear programming problems on {R}iemannian manifolds}, Pacific Journal of
  Optimization, 10 (2014), pp.~415--434.

\end{thebibliography}
\bibliographystyle{siam}

\appendix

\section{Proof of Lemma \ref{lemma:Lip-Rie-Hess}} \label{appendix-B}
For ease of notation, we denote $B_x = P_x\circ \Hess f(x)\circ P_x$, where $P_x$ is the orthogonal projection onto the tangent space $\cT_x\cM$. Then $\lambda_{\min}^{\cM}(\Hess f(x))$ corresponds to smallest eigenvalue of $B_x$ among the eigenvalues whose eigenvectors lie within $\cT_x\cM,$ which we denote as $\lambda_{\min}^\cM(B_x)$ for simplicity. We now prove \eqref{lemma:Lip-Rie-Hess-1} by the following steps.

\textbf{[Step 1.]} {\em To show}: There exists a finite closed covering for $\cM$ such that $\cM\subset\cup_{i=1}^{n_0}U_i$. Each $U_i$ is a closed ball in $\cE$ and there exists a local equation characterization of $\cM\cap U_i$, i.e., $x\in\cM\cap U_i$ if and only if $\phi^{(i)}(x)=0, x\in U_i$.

{\em Proof}. First, for any $x\in\cM$, there exists an open ball within which a local equation characterization is valid. Hence the union of these balls provide an open covering for $\cM$. By the compactness of $\cM$, there exists a finite open covering $\cup_{i=1}^{n_0}\bar{U}_i\supset\cM$ with $\bar{U}_i = B(x_i,r_i)$, where $B(x,r)$ denotes an open ball centered at $x$ with radius $r$.
Define $S = \partial (\cup_{i=1}^{n_0}\bar{U}_i)$ be the boundary of $\cup_{i=1}^{n_0}\bar{U}_i$, then we have
$\epsilon := \inf_{z\in\cM,y\in S}\|z-y\| > 0.$
This is because both $\cM$ and $S$ are compact sets and $\|\cdot\|$ is continuous, the infimum is attained at some $z_0$ and $y_0$. However, since $\cup_{i=1}^{n_0}\bar{U}_i$ is an open covering for $\cM$, $z_0\neq y_0$, hence $\epsilon = \|z_0-y_0\|>0$. Therefore, if we let $U_i = \mbox{\rm cl}(B(x_i,r_i-\half\epsilon))$ where $\mbox{\rm cl}(\cdot)$ denotes the closure of a set, then we have that
$\cup_{i=1}^{n_0}U_i$ is a finite closed covering for $\cM$. Within each $U_i$, a local equation characterization $\phi^{(i)}$ exists as given in the description of
%and is endowed from the one of
$\bar{U}_i$.

\textbf{[Step 2.]} {\em To show}: The inequality \eqref{lemma:Lip-Rie-Hess-1} holds within each $U_i$ with constant $D_i$.

{\em Proof}. With the local equation $\phi^{(i)}(x) = 0$, according to the results in \cite{RieOpt:Yang-etal-2012}, one can write
$B_x = P_x(\nabla^2 f(x) - \sum_{j=1}^{n-d}\mu_j(x)\nabla^2 \phi^{(i)}_j(x))P_x,$
where we have 
$P_x = I - \J\phi^{(i)}(x)^\top[\J\phi^{(i)}(x)\J\phi^{(i)}(x)^\top]^{-1}\J\phi^{(i)}(x),$ 
$\mu(x) = [\J\phi^{(i)}(x) \J\phi^{(i)}(x)^\top]^{-1}\J\phi^{(i)}(x)\nabla f(x).$
By the nonsingularity of $\J\phi^{(i)}(x)$ and compactness of $U_i\cap\cM$, the smoothness of $\phi^{(i)}$ indicates that $P_x$ and $\mu(x)$ are all smooth function in $x$, by the Lipschitz continuity of $\nabla f(x)$ and $\nabla^2f(x)$ and the compactness of $U_i$, we conclude that $B_x$ is Lipschitz continuous on $U_i$. Then there exists a $c_i$ such that
$$\|B_x-B_y\|_F \leq c_i\|x-y\|\mbox{ for } \forall x,y\in U_i\cap\cM.$$
Now we note the following well-known result on the continuity of %introduce the following existing result  on eigenvalue continuity
the eigenvalues of symmetric matrices.
\begin{lemma}\label{lemma:perturb}(Corollary 6.3.8, in \cite{Book-Matrix-2012} on page 407)
	Let $A,E\in\mathbb{R}^{n\times n}$. Assume that $A, E$ are symmetric. Let $\lambda_1\leq\lambda_2\leq\cdots\leq\lambda_n$ be the eigenvalues of $A$, and let $\hat{\lambda}_1\leq \hat{\lambda}_2\leq\cdots\leq\hat{\lambda}_n$ be the eigenvalues of $A+E$. Then we have $\sum_{j = 1}^n\|\lambda_j-\hat{\lambda}_j\|^2\leq\|E\|_F^2.$
\end{lemma}
Note that the eigenvalues of $B_x$ are $n-d$ zeros and $d$ eigenvalues with corresponding eigenvectors lying within $\cT_x\cM$. Now consider the matrices $B_x$ and $B_y$ in the following three cases. First, when both $\lambda_{\min}^{\cM}(B_x), \lambda_{\min}^{\cM}(B_y)\leq 0$, we have $\lambda_{\min}^{\cM}(B_x) = \lambda_{1}(B_x)$, $\lambda_{\min}^{\cM}(B_y) = \lambda_{1}(B_y)$. Then applying  Lemma \ref{lemma:perturb} we have
$$|\lambda_{\min}^{\cM}(B_x)-\lambda_{\min}^{\cM}(B_y)|\leq \sqrt{\sum_{j = 1}^n(\lambda_j(B_x)-\lambda_j(B_y))^2}\leq \|B_x-B_y\|_F\leq c_i\|x-y\|.$$
Second, when both $\lambda_{\min}^{\cM}(B_x)\geq 0$ and $\lambda_{\min}^{\cM}(B_y)\geq 0,$ we have $\lambda_{\min}^{\cM}(B_x) = \lambda_{n-d+1}(B_x)$ and $\lambda_{\min}^{\cM}(B_y) = \lambda_{n-d+1}(B_y)$. Then the same argument of the first case goes through similarly.
Third, when $\lambda_{\min}^{\cM}(B_x)$ and $\lambda_{\min}^{\cM}(B_y)$ have different signs, e.g., $\lambda_{\min}^{\cM}(B_x)<0$ while $\lambda_{\min}^{\cM}(B_y)> 0$, we have the following arguments,
$$\lambda_{1}(B_x)=\lambda_{\min}^{\cM}(B_x)<0,~\lambda_{n-d+1}(B_x)\leq 0,$$ $$\lambda_{1}(B_y) = 0, ~\lambda_{n-d+1}(B_y)=\lambda_{\min}^{\cM}(B_y)>0.$$
Applying Lemma \ref{lemma:perturb} in a similar way as before we have
$$-c_1\|x-y\|\leq \lambda_{\min}^{\cM}(B_x)<0, \mbox{ and } 0<\lambda_{\min}^{\cM}(B_y)\leq c_1\|x-y\|+\lambda_{n-d+1}(B_x)\leq c_1\|x-y\|.$$
Hence $|\lambda_{\min}^{\cM}(B_x)-\lambda_{\min}^{\cM}(B_y)|\leq 2c_1\|x-y\|$. Therefore if we take $D_i= 2c_i$, then the statement of Step 2 is proved.

\textbf{[Step 3.]} {\em To show}: There exists a constant $d_{\cM}>0$ such that for $x,y\in\cM$, if $\|x-y\|\leq d_{\cM}$ then there exists a $U_i$ from the finite closed cover of $\cM$ such that $x,y\in \mbox{\rm int}(U_i)$, where $\mbox{\rm int}(\cdot)$ indicates the interior of a set.

{\em Proof}. First, by our construction of the $U_i$'s, we known that $\cup_{i=1}^{n_0}\mbox{\rm int}(U_i)$ is also a finite open cover of $\cM$. Suppose the statement is not true, then there exists a sequence $\{x_k,y_k\}\subset\cM$ such that the pair $x_k,y_k$ does not belong to the interior of a same $U_i$ for $\forall k,\forall i$, but $\|x_k,y_k\|\rightarrow 0$. By the compactness of $\cM$, we have convergent subsequence $\{x_{k_r},y_{k_r}\}$ such that both $x_{k_r}\rightarrow z_0, y_{k_r}\rightarrow z_0$. Then this indicates that $z_0\notin\cup_{i=1}^{n_0}\mbox{\rm int}(U_i),$
otherwise for sufficiently large $r$, $x_{k_r},y_{k_r}$ shall lie in some same $\mbox{\rm int}(U_i)$ and hence yields a contradiction. Therefore $z_0\notin\cM$, but this also contradicts the compactness of $\cM$ which infers that $z_0\in\cM.$

\textbf{[Step 4.]} Combining all the previous results and letting $D = \max_{1\leq i\leq n_0}D_i$, %proves the general statement of the theorem.
the conclusion of the theorem follows.
In the special case of Stiefel Manifolds, which can be characterized by %the whole manifold is characterized with a same
a smooth and uniform equation $X^\top X=I$, the finite covering arguments in Step 1 is unnecessary; $\cM$ itself is a valid closed covering. In that case, $d_{\cM}$ can be set to $+\infty.$

\section{Proof of Proposition \ref{lemma:Stief-pullback-Lips}} \label{Appendix-C}
Here we consider the extended polar retraction on the Stiefel manifold $\St_{n,r}$, written as $\Retr(X,Z) = (X+Z_P)(I+Z_P^\top Z_P)^{-\half}$,
where $Z_P:=P_x[Z]$ is the orthogonal projection of $Z$ onto $\cT_X\St_{n,r}$. However, since $Z$ and $V$ are already in the tangent space $\cT_X\St_{n,r}$, we drop the subscript $P$ in the subsequent discussion. This extended version is applied to enable the usage of Euclidean calculus tools. For the ease of notation, we define
$Y_{X,Z,V}(t) = \Retr(X,Z+tV)$ and  $g_{X,Z,V}(t) = f( Y_{X,Z,V}(t)).$ Therefore,
$$\langle (\nabla_{\xi}^2\hat{f}(Z)-\nabla_{\xi}^2\hat{f}(0) )[V],V\rangle = g_{X,Z,V}''(0)-g_{X,0,V}''(0). $$
This suggests that to prove \eqref{pullback-Lips} it suffices to prove
\begin{equation}
\label{thm:pb-Lips-1}
|g_{X,Z,V}''(0)-g_{X,0,V}''(0)|\leq L_H\|Z\|_F, \,\, \forall V\in\cT_X\St_{n,r}, \|V\|_F=1,
\end{equation}
with an estimation of $L_H$. By direct computation,
$$g''_{X,Z,V}(0) = \langle \nabla f(Y_{X,Z,V}(0)),Y''_{X,Z,V}(0)\rangle + \langle \nabla^2f(Y_{X,Z,V}(0))[Y'_{X,Z,V}(0)], Y'_{X,Z,V}(0)\rangle.$$
%As a consequence, directly
Applying the triangular inequality yields
\begin{small}
	\begin{eqnarray*}
		|g_{X,Z,V}''(0)-g_{X,0,V}''(0)| & \leq & |\langle \nabla f(Y_{X,Z,V}(0)),Y''_{X,Z,V}(0)-Y''_{X,0,V}(0)\rangle| \\
		&& + |\langle \nabla f(Y_{X,Z,V}(0))-\nabla f(Y_{X,0,V}(0)),Y''_{X,0,V}(0)\rangle| \\
		& & + \|Y'_{X,Z,V}(0)\|_F\|\nabla^2f(Y_{X,Z,V}(0))\|_F\|Y'_{X,Z,V}(0)-Y'_{X,0,V}(0)\|_F  \\
		& &+ \|Y'_{X,Z,V}(0)\|_F\|\nabla^2f(Y_{X,Z,V}(0))-\nabla^2f(Y_{X,0,V}(0))\|_F\|Y'_{X,0,V}(0)\|_F \\
		& &+
		\|Y'_{X,Z,V}(0)-Y'_{X,0,V}(0)\|_F\|\nabla^2f(Y_{X,0,V}(0))\|_F\|Y'_{X,0,V}(0)\|_F.
\end{eqnarray*}\end{small}Recalling the nature of the parameters  $G,\ell_f,\ell_H$, and Proposition \ref{prop:retraction-regularity}, the above inequality can be simplified to
\begin{eqnarray}
\label{thm:pb-Lips-2}
|g_{X,Z,V}''(0)-g_{X,0,V}''(0)| & \leq &  G\|Y''_{X,Z,V}(0)-Y''_{X,0,V}(0)\|_F + \ell_fL_1\|Y''_{X,0,V}(0)\|_F\|Z\|_F \nonumber\\
& & + \ell_f(\|Y'_{X,Z,V}(0)\|_F+\|Y'_{X,0,V}(0)\|_F)\|Y'_{X,Z,V}(0)-Y'_{X,0,V}(0)\|_F  \\
& &+ \ell_HL_1\|Y'_{X,Z,V}(0)\|_F\|Y'_{X,0,V}(0)\|_F \|Z\|_F.\nonumber
\end{eqnarray}
Thus we need only to bound the following terms by
\begin{equation}
\label{thm:pb-1}
\|Y''_{X,Z,V}(0)-Y''_{X,0,V}(0)\|_F\mbox{ and } ~~\|Y'_{X,Z,V}(0)-Y'_{X,0,V}(0)\|_F\leq \mathcal{O}(\|Z\|_F),
\end{equation}
\begin{equation}
\label{thm:pb-2}
\|Y'_{X,0,V}(0)\|_F, \|Y''_{X,0,V}(0)\|_F \mbox{ and }\|Y'_{X,Z,V}(0)\|_F = \mathcal{O}(1).
\end{equation}
Now we bound these terms in the following steps.

\textbf{[Step 1.]} First we characterize the derivatives of $Y_{X,Z,V}(t)$. Define
$S_Z := I_r+Z^\top Z $ and define
\begin{equation*}
F_{Z,V}(t) := (I_r+(Z+tV)^\top(Z+tV))^{-\half} = (S_Z + t(Z^\top V+V^\top Z+tV^\top V))^{-\half}.
\end{equation*}
Then we have $Y_{X,Z,V}(t) = (X+Z+tV)F_{X,Z}(t)$ and $
F_{Z,V}^2(t) = (I_r+tS_Z^{-1}(V^\top Z+Z^\top V+tV^\top V))^{-1}S_Z^{-1}.$ Let the expansion of $F_{Z,V}(t)$ be
\begin{equation}
\label{thm:pb-Lips-3}
F_{Z,V}(t) = F_{Z,V}(0) + tC_{Z,V} + t^2D_{Z,V} + \mathcal{O}(t^3)
\end{equation}
and let us calculate the expansion of $F_{Z,V}^2(t)$ by
\begin{small}
	\begin{eqnarray}
	\label{thm:pb-Lips-4}
	F_{Z,V}^2(t) = [I_r - tS_Z^{-1}(V^\top Z+Z^\top V+tV^\top V)+(tS_Z^{-1}(V^\top Z+Z^\top V+tV^\top V))^2]S_Z^{-1}+\mathcal{O}(t^3)\nonumber
	\end{eqnarray}
\end{small}
where we use the formula $(I+A)^{-1} = I+\sum_{i=1}^\infty(-1)^iA^i$ when $\|A\|<1$. Then, by comparing the coefficients of $t$ and $t^2$ terms between the above expansion of $F_{Z,V}^2(t)$ and \eqref{thm:pb-Lips-3}, $F_{Z,V}(t)^2 = (F_{Z,V}(0) + tC_{Z,V} + t^2D_{Z,V})^2 + \mathcal{O}(t^3),$
we have
\begin{equation}
\label{thm:pb-Lips-5}
S_Z^{-\half}C_{Z,V} + C_{Z,V}S_Z^{-\half} = -S_Z^{-1}(V^\top Z+Z^\top V)S_Z^{-1},
\end{equation}
and
\begin{eqnarray}
\label{thm:pb-Lips-6}
&&S_Z^{-\half}D_{Z,V} + D_{Z,V}S_Z^{-\half}+C_{Z,V}^2 \\
&=& -S_Z^{-1}(V^\top V)S_Z^{-1} + S_Z^{-1}(V^\top Z+Z^\top V)S_Z^{-1}(V^\top Z+Z^\top V)S_Z^{-1}.\nonumber
\end{eqnarray}
With $C_{Z,V}$ and $D_{Z,V}$ in place, we can write the derivatives of $Y_{X,Z,V}(t)$ explicitly as
\begin{equation}
\label{thm:pb-Lips-7}
\begin{cases}
Y'_{X,Z,V}(0) = (X+Z)C_{Z,V} + VS_Z^{-\half},\\
Y''_{X,Z,V}(0) = 2VC_{Z,V} + 2(X+Z)D_{Z,V}.
\end{cases}
\end{equation}
Note that when $Z = 0$ and $S_Z = I_r$, we can solve \eqref{thm:pb-Lips-5} and \eqref{thm:pb-Lips-6} to yield $S_0 = I, C_{0,V} = 0, D_{0,V} = -\half V^\top V$. Consequently,
\begin{equation}
\label{thm:pb-Lips-8}
Y'_{X,0,V}(0) = V\mbox{ and } Y''_{X,0,V}(0) = -XV^\top V.
\end{equation}

\textbf{[Step 2.]} Bound the term $\|Y''_{X,Z,V}(0)-Y''_{X,0,V}(0)\|_F\leq \mathcal{O}(\|Z\|_F)$ for $\forall V\in\cT_X\St_{n,r}, \|V\|_F \leq 1$.
By  \eqref{thm:pb-Lips-7} and \eqref{thm:pb-Lips-8}, we have
\begin{eqnarray}
\label{thm:pb-Lips-9}
\|Y''_{X,Z,V}(0)-Y''_{X,0,V}(0)\|_F & = & \|2VC_{Z,V}+2(X+Z)D_{Z,V}+XV^\top V\|_F\nonumber\\
&\leq& \underbrace{2\|C_{Z,V}\|_F}_{T_1} + \underbrace{\|2XD_{Z,V}+XV^\top V\|_F}_{T_2} + \underbrace{2\|D_{Z,V}\|_F\|Z\|_F}_{T_3} .
\end{eqnarray}
First, let us consider the term $T_1$. Since $S_Z\succeq I_r$, we have $\|C_{Z,V}\|_F\leq \|S_Z^{\half}C_{Z,V}\|_F$. Thus we choose to  bound $\|S_Z^{\half}C_{Z,V}\|_F$ which will be useful later. If we denote by $\vvec(X)$ the vectorization operator for a matrix $X$, then a handy formula gives $\vvec(AXB) = (B^\top\otimes A) \vvec(X)$. Note that by using the $\vvec(\cdot)$ operator, equation \eqref{thm:pb-Lips-5} has an explicit solution
\begin{eqnarray*}
	\vvec(C_{Z,V}) & = & (I_r\otimes S_Z^{-\half}+S_Z^{-\half}\otimes I_r)^{-1}\vvec(S_Z^{-1}(V^\top Z+Z^\top V)S_Z^{-1})\\
	& = & (I_r\otimes S_Z^{-\half}+S_Z^{-\half}\otimes I_r)^{-1}(S_Z^{-1}\otimes  S_Z^{-1})\vvec(V^\top Z+Z^\top V)\\
	%	& = &  (I_r\otimes S_Z^{-\half}+S_Z^{-\half}\otimes I_r)^{-1}(S_Z\otimes  S_Z)^{-1}vec(V^\top Z+Z^\top V) \\
	& = &  (S_Z\otimes S_Z^{\half}+S_Z^{\half}\otimes S_Z)^{-1} \vvec(V^\top Z+Z^\top V).
\end{eqnarray*}
Therefore
\begin{eqnarray}
\label{thm:pb-Lips-9.2}
\vvec(S_Z^{\half}C_{Z,V}) & = &(I_r\otimes S_Z^{\half})\vvec(C_{Z,V})\\
& = &(I_r\otimes S_Z^{\half} + S_Z^{\half}\otimes S_Z^{\half})^{-1}\left[\vvec(V^\top Z)+\vvec(Z^\top V)\right].\nonumber
\end{eqnarray}
Since $S_Z\succeq I_r, S_Z^{\half}\succeq I_r$, we have $I_r\otimes S_Z^{\half} + S_Z^{\half}\otimes S_Z^{\half}  \succeq 2I_{r^2}$, and therefore
\begin{equation}
\label{thm:pb-Lips-9.5}
\|(I_r\otimes S_Z^{\half} + S_Z^{\half}\otimes S_Z^{\half})^{-1}\|_2\leq \half,
\end{equation} 
\begin{eqnarray}
\label{thm:pb-Lips-10}
\|S_Z^{\half}C_{Z,V}\|_F & = &\|\vvec(S_Z^{\half}C_{Z,V})\|_F \leq \half\|\vvec(V^\top Z+Z^\top V)\|_F\\& \leq &\half(\|V^\top Z\|_F+\|Z^\top V\|_F)\leq\|Z\|_F\nonumber
\end{eqnarray}
where the last inequality is due to $\|V\|_F\leq 1$. Hence
\begin{equation}
\label{thm:pb-Lips-11}
T_1\leq 2\|Z\|_F.
\end{equation}

Now for the benefit of discussion later, %for the ease of later on discussion,
%we continue to dig into the bound of
let us further bound
$\|S_Z^\half C_{Z,V}\|_F$ by a constant. Let the SVD of $Z$ be $Z = Q\Lambda^\half U^\top.$ Then $Z^\top Z = U\Lambda U^\top$, $S_Z = U(I_r+\Lambda)U^\top$ and $S_Z^\half = U(I_r+\Lambda)^\half U^\top$. Hence we have
$$(I_r\otimes S_Z + S_Z^\half\otimes S_Z^\half)^{-1} = (U\otimes U)(I_r\otimes(\Lambda+I_r)+(\Lambda+I_r)^\half\otimes(\Lambda+I_r)^\half)(U\otimes U)^\top.$$
Therefore, instead of bounding $\|(I_r\otimes S_Z + S_Z^\half\otimes S_Z^\half)^{-1}\vvec(Z^\top V)\|_F$ in \eqref{thm:pb-Lips-10} by means of \eqref{thm:pb-Lips-9.5},
%we bound it in the following way,
we now take a different approach:
\begin{eqnarray*}
	\label{thm:pb-Lips-11.5}
	& &\|(I_r\otimes S_Z + S_Z^\half\otimes S_Z^\half)^{-1}\vvec(Z^\top V)\|_F \\
	& = & \|(I_r\otimes S_Z + S_Z^\half\otimes  S_Z^\half)^{-1}\vvec(U\Lambda^\half Q^\top VI_r)\|_F  \nonumber\\
	& = & \|(U\otimes U)(I_r\otimes(\Lambda+I_r)+(\Lambda+I_r)^\half\otimes(\Lambda+I_r)^\half)^{-1}(U^\top \otimes U^\top)(I_r\otimes U\Lambda^\half)\vvec(Q^\top V)\|_F\nonumber\\
	& \leq & \|U\otimes U\|_2\|(I_r\otimes(\Lambda+I_r)+(\Lambda+I_r)^\half\otimes(\Lambda+I_r)^\half)^{-1}(I_r\otimes\Lambda^\half)\|_2\|U^\top\otimes I_r\|_2\|Q\|_2\|V\|_F . \nonumber
\end{eqnarray*}
Note that $\|U\otimes U\|_2 = \|U^\top\otimes I_r\|_2 = \|Q\|_2 = 1,$ $\|V\|_F\leq 1$. If we denote $\lambda_i$ as the $i$th diagonal element of $\Lambda$, then all the eigenvalues of the diagonal matrix $(I_r\otimes(\Lambda+I_r)+(\Lambda+I_r)^\half\otimes(\Lambda+I_r)^\half)^{-1}(I_r\otimes\Lambda^\half)$ can be written as
$$\sigma_{ij} = \frac{\lambda_j^\half}{1+\lambda_j+(1+\lambda_i)^\half(1+\lambda_j)^\half}\leq \half.$$
Hence we end up with $\|(I_r\otimes S_Z + S_Z^\half\otimes S_Z^\half)^{-1}\vvec(Z^\top V)\|_F\leq \half.$ Similarly, we have  $\|(I_r\otimes S_Z + S_Z^\half\otimes S_Z^\half)^{-1}\vvec(V^\top Z)\|_F\leq \half.$ Together with \eqref{thm:pb-Lips-9.2}, they give an alternative bound which states $\|S_Z^\half C_{Z,V}\|_F\leq 1.$ In total,
\begin{equation}
\label{thm:pb-Lips-12}
\|C_{Z,V}\|_F\leq \|S_Z^\half C_{Z,V}\|_F\leq \min\{\|Z\|_F,1\}.
\end{equation}

Second, we now come to bound the term $T_2$. Reformulating \eqref{thm:pb-Lips-6} slightly, we have
$$D_{Z,V}S_Z^\half + S_Z^\half D_{Z,V} = -S_Z^\half C_{Z,V}^2S_Z^\half  - S_Z^{-\half}V^\top VS_Z^{-\half} + \left(S_Z^{-\half}(Z^\top V+V^\top Z)S_Z^{-\half}\right)^2.$$
Let us define $H_{Z,V}$ be the matrix that satisfies
\begin{equation}
\label{thm:pb-Lips-13}
H_{Z,V}S_Z^\half + S_Z^\half H_{Z,V} = S_Z^{-\half}V^\top VS_Z^{-\half},
\end{equation}
and define $J_{Z,V} = D_{Z,V}+H_{Z,V}$. Then,
\begin{eqnarray}
T_2 &=& \|2D_{Z,V}+V^\top V\|_F \\& \leq & 2\|D_{Z,V}+H_{Z,V}\|_F+\|2H_{Z,V}-V^\top V\|_F \nonumber  \\& = & 2\|J_{Z,V}\|_F+\|2H_{Z,V}-V^\top V\|_F. \label{thm:pb-Lips-13.5}\nonumber
\end{eqnarray}
Note that by the definition of $J_{Z,V}$ and $H_{Z,V}$,
$$J_{Z,V}S_Z^\half + S_Z^\half J_{Z,V} = -S_Z^\half C^2_{Z,V}S_Z^\half + \left(S_Z^{-\half}(Z^\top V+V^\top Z)S_Z^{-\half}\right)^2.$$
Similar to the bound for $T_1$, we have
\begin{eqnarray}
\|J_{Z,V}\|_F &=& \|\vvec(J_{Z,V})\|_F \nonumber \\
&\leq& \left\|\left(I_r\otimes S_Z^\half+ S_Z^\half\otimes I_r\right)^{-1}\right\|_2\left(\|S_Z^\half C_{Z,V}\|_F^2 + \|S_Z^{-\half}(Z^\top V+V^\top Z)S_Z^{-\half}\|_F^2\right). \label{thm:pb-Lips-14}
\end{eqnarray}
Note that
\begin{eqnarray*}
	\|S_Z^{-\half}(Z^\top V+V^\top Z)S_Z^{-\half}\|_F \leq 2\|S_Z^{-\half}V^\top ZS_Z^{-\half}\|_F\leq 2\|S_Z^{-\half}\|_2\|ZS_Z^{-\half}\|_2\|V\|_F\leq 2\|ZS_Z^{-\half}\|_2,
\end{eqnarray*}
where the last inequality is due to $\|S_Z^{-\half}\|_2,\|V\|_F\leq 1$. By the SVD of $Z$,
$$\|ZS_Z^{-\half}\|_2 = \|Q\Lambda^\half U^\top U(I_r+\Lambda)^{-\half}U^\top \|_2 = \| \Lambda^\half  (I_r+\Lambda)^{-\half} \|_2\leq 1,$$
and so
$$\|S_Z^{-\half}(Z^\top V+V^\top Z)S_Z^{-\half}\|_F\leq 2.$$
On the other hand,
$$\|S_Z^{-\half}(Z^\top V+V^\top Z)S_Z^{-\half}\|_F\leq2\|S_Z^{-\half}\|_2^2\|V\|_F\|Z\|_F\leq2\|Z\|_F.$$
Therefore,
\begin{equation}
\label{thm:pb-Lips-15}
\|S_Z^{-\half}(Z^\top V+V^\top Z)S_Z^{-\half}\|_F^2 \leq \min\{4,4\|Z\|_F\}.
\end{equation}
Similarly, \eqref{thm:pb-Lips-12} indicates that
$$\|S_Z^{ \half}C_{Z,V} \|_F^2\leq \min\{1,\|Z\|_F\}.$$
Together with \eqref{thm:pb-Lips-9.5}, the above bounds and \eqref{thm:pb-Lips-14}, we have
\begin{equation}
\label{thm:pb-Lips-16}
\|J_{Z,V}\|_F\leq \min\left\{\frac{5}{2},\frac{5}{2}\|Z\|_F\right\}.
\end{equation}

For  $\|2H_{Z,V}-V^\top V\|_F$, let us start with the explicit solution of equation \eqref{thm:pb-Lips-13}, which is
$$2\vvec(H_{Z,V}) = 2(S_Z^{ \half}\otimes S_Z + S_Z\otimes S_Z^{ \half})^{-1}\vvec(V^\top V)$$
leading to
\begin{equation}
\label{thm:pb-Lips-16.5}
\|H_{Z,V}\|_F\leq \half.
\end{equation}
and
\begin{eqnarray*}
	\|2\vvec(H_{Z,V})-\vvec(V^\top V)\|_F &\leq& \|I_{r^2} - 2(S_Z^{ \half}\otimes S_Z + S_Z\otimes S_Z^{ \half})^{-1}\|_2\|\vvec(V^\top V)\|_F \\
	& \leq & 1-\frac{2}{\lambda_{\max}(S_Z^{ \half}\otimes S_Z + S_Z\otimes S_Z^{ \half})}\\
	& \leq & 1-\frac{1}{\lambda_{\max}(  S_Z\otimes S_Z^{ \half})}\\
	& = & 1-\frac{1}{\lambda_{\max}^{\frac{3}{2}}(  S_Z )}\\
	& \leq &  1-\frac{1}{(1+\|Z\|_F^2)^{\frac{3}{2}}} .
\end{eqnarray*}
Let $w = \|Z\|_F$, and define the function $q(w) = 1-(1+w^2)^{-\frac{3}{2}}, w\geq0.$ Then it is easy to prove that
$q(w)\leq 0.66w$ for all $w\geq0$. Consequently,
\begin{equation}
\label{thm:pb-Lips-17}
\| H_{Z,V} - V^\top V \|_F = \|2\vvec(H_{Z,V})-\vvec(V^\top V)\|_F \leq 0.66\|Z\|_F.
\end{equation}
Therefore, combining \eqref{thm:pb-Lips-13.5}, \eqref{thm:pb-Lips-16} and \eqref{thm:pb-Lips-17} yields
\begin{equation}
\label{thm:pb-Lips-18}
T_2\leq 2\cdot\frac{5}{2}\|Z\|_F+0.66\|Z\|_F = 5.66\|Z\|_F.
\end{equation}

Third, we bound the term $T_3$ by
\begin{equation}
\label{thm:pb-Lips-19}
T_3  = 2\|D_{Z,V}\|_F\|Z\|_F
\leq  2(\|J_{Z,V}\|_F+\|H_{Z,V}\|_F) \leq 6\|Z\|_F,
\end{equation}
where the last inequality is due to \eqref{thm:pb-Lips-16} and \eqref{thm:pb-Lips-16.5}. Now combining the bounds on $T_1,T_2$ and $T_3$, we finally finish the Step 2 with
\begin{equation}
\label{thm:pb-Lips-20}
\|Y''_{X,Z,V}(0) - Y''_{X,0,V}(0)\|_F \leq T_1+T_2+T_3 \leq 13.66\|Z\|_F.
\end{equation}

\textbf{[Step 3.]} Bound the term  $\|Y'_{X,Z,V}(0) - Y'_{X,0,V}(0)\|_F$ by
\begin{eqnarray*}
	\|Y'_{X,Z,V}(0) - Y'_{X,0,V}(0)\|_F & = & \|(X+Z)C_{Z,V}+VS_Z^{-\half}-V\|_F \\
	& \leq &\|X\|_2\|C_{Z,V}\|_F+\|Z\|_F\|C_{Z,V}\|_F + \|I_r-S_Z^{-\half}\|_2\|V\|_F \\
	& \leq & \|Z\|_F + \|Z\|_F + (1-\lambda_{\max}^{-\half}(S_Z))\\
	& \leq & 2\|Z\|_F + (1-(1+\|Z\|_F^2)^{-\half})\\
	& \leq & 2.31\|Z\|_F.
\end{eqnarray*}
The second inequality is due to $\|X\|_2=1, \|V\|_F\leq 1, \|C_{Z,V}\|_F\leq\min\{1,\|Z\|_F\}$, the last inequality is due to $(1-(1+w^2)^{-\half})\leq 0.31w, \, \forall w\geq0.$

\textbf{[Step 4.]} Bound the terms $\|Y'_{X,Z,V}(0)\|_F, \|Y'_{X,0,V}(0)\|_F\mbox{ and } \|Y''_{X,0,V}(0)\|_F.$
By previous results, directly, we have
$$\|Y'_{X,0,V}(0)\|_F = \|V\|_F \leq 1,\quad \|Y''_{X,0,V}(0)\|_F = \|XV^\top V\|_F\leq 1,$$
and
\begin{eqnarray*}
	\|Y'_{X,Z,V}(0)\|_F & = & \|(X+Z)C_{Z,V}+VS_Z^{-\half}\|_F \leq 2+\|ZC_{Z,V}\|_F.
\end{eqnarray*}
%Similar to proving %the tricks applied in
%\eqref{thm:pb-Lips-11.5}, we still use 
Again, by the SVD of $Z$ we obtain
\begin{eqnarray*}
	\|\vvec(ZC_{Z,V}I)\|_F  & = &   \|(I_r\otimes Z)\vvec(C_{Z,V})\|_F \\ & \leq & \|(I_r\otimes Z)(S_Z^\half\otimes S_Z+S_Z\otimes S_Z^\half)^{-1}\vvec(V^\top Z+Z^\top V)\|_F.
\end{eqnarray*}
Note that
\begin{eqnarray*}
	& & \|(I_r\otimes Z)(S_Z^\half\otimes S_Z+S_Z\otimes S_Z^\half)^{-1}\vvec(Z^\top V)\|_F \\
	& = & \|(I_r\otimes Q\Lambda^\half U^\top)(U\otimes U)(\Lambda^\half\otimes\Lambda+\Lambda\otimes\Lambda^\half)^{-1}(U^\top\otimes U^\top)(I_r\otimes U\Lambda^\half Q^\top)\vvec(V)\|_F\\
	& \leq & \|(U\otimes Q)\|_2\|(I_r\otimes \Lambda^\half)(\Lambda^\half\otimes\Lambda+\Lambda\otimes\Lambda^\half)^{-1}(I_r\otimes  \Lambda^\half )\|_2\|(U^\top\otimes Q^\top)\|_2\\
	&\leq & \max_{i,j}\frac{\lambda_j}{1+\lambda_j+(1+\lambda_j)^\half(1+\lambda_i)^\half}\leq 1.
\end{eqnarray*}
Similarly, $\|(I_r\otimes Z)(S_Z^\half\otimes S_Z+S_Z\otimes S_Z^\half)^{-1}\vvec(V^\top Z)\|_F\leq 1$. Consequently
$$\|Y'_{X,Z,V}(0)\|_F\leq  2+\|\vvec(ZC_{Z,V}I)\|_F\leq 4.$$
Now with all the bounds for \eqref{thm:pb-1} and \eqref{thm:pb-2}, and $L_1=1$, we derive a bound for \eqref{thm:pb-Lips-2} as
\begin{eqnarray*}
	|g_{X,Z,V}''(0)-g_{X,0,V}''(0)| & \leq &  L_H \|Z\|_F,
\end{eqnarray*}
where $L_H = 13.66G  + 12.55\ell_f + 4\ell_H.$

\section{Proof of Proposition \ref{lemma:Stief-pullback-grad-Neibourhood}} \label{Appendix-D}
Due to Proposition \ref{prop:pullback-grad-Neighbourhood}, we only need to bound the difference $\|P_W-\J_{\xi}\Retr(X,Z)\|_2$, where $W = \Retr(X,Z)$. It suffices to show for any $V\in\R^{n\times r}$ and $\|V\|_F=1$ that
$$\|P_W[V]-\J_{\xi}\Retr(X,Z)[V]\|_F = \|P_W[V]-
Y'_{X,Z,V}(0)\|_F\leq \mathcal{O}(\|Z\|_F), \, \forall Z\in\cT_X\St_{n,r},$$
where $Y_{X,Z,V}(t)$ is used in the proof of Proposition \ref{lemma:Stief-pullback-Lips}.
%If we also inherit the definitions of $S_Z$ and $C_{Z,V}$.
By \eqref{thm:pb-Lips-7} we have
\begin{eqnarray}
\label{thm:puba-gd-1}
\begin{cases}
Y'_{X,0,V}(0) = V_P := P_X[V],\\
Y'_{X,Z,V}(0) = (X+Z)C_{Z,V_P} + V_PS_Z^{-\half}.
\end{cases}
\end{eqnarray}
Note that in this case $V$ is not required to lie within $\cT_X\St_{n,r}$, so `$V$' in \eqref{thm:pb-Lips-7} should be changed to `$V_P$'. Observe,
\begin{eqnarray*}
	\|P_W[V]- Y'_{X,Z,V}(0)\|_F & \leq & \|Y'_{X,Z,V}(0)- Y'_{X,0,V}(0)\|_F + \|P_W[V]- Y'_{X,0,V}(0)\|_F\\
	& \leq & 2.31\|Z\|_F + \|P_W[V]-P_X[V]\|_F ,
\end{eqnarray*}
where the second inequality is due to Step 3 in the proof of Proposition~\ref{lemma:Stief-pullback-Lips}. We have
\begin{eqnarray*}
	\|P_W[V]-P_X[V]\|_F & = & \half\|WW^\top V+WV^\top W-XX^\top V-XV^\top X\|_F \\
	& \leq & \half\|WW^\top V-WX^\top V\|_F + \half\|WX^\top V-XX^\top V\|_F \\ & &+ \half\|WV^\top W - WV^\top X\|_F + \half\|WV^\top X-XV^\top X\|_F \\
	%	& \leq & \half\|W\|_2\|W-X\|_F\|V\|_F + \half\|W-X\|_F\|X\|_2\|V\|_F  \\ &&+\half\|W\|_2\|V\|_F\|W-X\|_F + \half\|W-X\|_F\|X\|_2\|V\|_F \\
	& \leq &2\|W-X\|_F \leq  2L_1\|Z\|_F = 2\|Z\|_F .
\end{eqnarray*}

Hence $\|P_W[V]- Y'_{X,Z,V}(0)\|_F\leq 4.31\|Z\|_F$. As long as $\|Z\|_F\leq \frac{1}{8.62}$, we have
$\|P_W[V]- Y'_{X,Z,V}(0)\|_F\leq\half$,
$\forall X\in \St_{n,r}$, $\forall Z\in\cT_X\St_{n,r}$, $\forall \|V\|\leq 1$, and consequently
$$\|P_W-\J_{\xi}\Retr(X,Z)\|_2\leq\half.$$
One last remark is that with some additional effort, this threshold on $\|Z\|_F$ can be slightly improved from $1/8.62$ to $1/8$. However, we shall leave it out here for simplicity.
%%%%%%%%%%%%%%%%%%%%%%%%%%%%%%%%%%%%%%%%%%%%%
% PLEASE DO NOT COMMENT THIS PARAGRAPH, OTHERWISE THE DISCUSSION DOES NOTCOMPLETE THE PROOF.
%Note that in the theorem the statement requires only  $\|Z\|_F\leq \frac{1}{8}$. This is because we bound the term $\|Y'_{X,Z,V}(0)- Y'_{X,0,V}(0)\|_F$ directly by the bound derived in the Step 3 of proof of Theorem \ref{theorem:pullback-Lips}. Actually with more effort, one can derive another bound that $\|Y'_{X,Z,V}(0)- Y'_{X,0,V}(0)\|_F\leq \|Z\|_F+2\|Z\|_F^2.$ With this bound, the $1/8$ threshold can be recovered. However, for simplicity, the intermediate steps are not shown here.
% PLEASE DO NOT COMMENT THIS PARAGRAPH, OTHERWISE THE DISCUSSION DOES NOTCOMPLETE THE PROOF.
%%%%%%%%%%%%%%%%%%%%%%%%%%%%%%%%%%%%%%%%%%%%%

\end{document}